\documentclass[12pt]{amsart}

\usepackage{amsmath,amsthm,amssymb,amsfonts,eucal,a4wide}

\usepackage[utf8]{inputenc}
\usepackage[english]{babel}
\usepackage[pagewise]{lineno} 
\usepackage{hyperref}

\newtheorem{thm}{Theorem}
\newtheorem{theorem}[thm]{Theorem}
\newtheorem{Corollary}[thm]{Corollary}
\newtheorem{lem}[thm]{Lemma}
\newtheorem{Lemma}[thm]{Lemma}
\newtheorem{prop}[thm]{Proposition}

\newtheorem{proposition}[thm]{Proposition}

\newtheorem{definition}[thm]{Definition}

\newtheorem{conjecture}[thm]{Conjecture}

\newcommand{\LL}{{\mathbb{L}}}
\newcommand{\ch}{{\text{ch}}}

\newcommand{\Q}{\mathbb{Q}}
\newcommand{\Fr}{\mathfrak{Fr}}
\newcommand{\CC}{\mathbb{C}}

\newcommand{\An}{\mathcal{A}_{n}}

\newcommand{\Hr}{\mathrm{H}}

\newcommand{\ZZ}{\mathbb{Z}}

\newcommand{\OO}{\mathcal{O}}
\newcommand{\AAA}{\mathcal{A}}
\newcommand\II{{\mathcal I}}
\newcommand{\DT}{\mathsf{PT}}
\newcommand{\PT}{\mathsf{PT}}

\newcommand{\GW}{\mathsf{GW}}
\newcommand{\proj}{\mathbf{P}}
\newcommand{\com}{\mathbb{C}}

\begin{document}

%\linenumbers
%\setpagewiselinenumbers

\title{GW/PT descendent correspondence  via vertex operators}
\author{A. Oblomkov, A. Okounkov, \and R. Pandharipande}
\newcommand{\Addresses}{{% additional braces for segregating \footnotesize
  \bigskip
  \footnotesize

  A.~Oblomkov, \textsc{Department of Mathematics and Statistics, Univ. of Massachusetts,
    Amherst}\par\nopagebreak
  \textit{E-mail address:} \texttt{oblomkov@math.umass.edu}

  \medskip

  A.~Okounkov, \textsc{Department of Mathematics,
    Columbia University, New York}\par\nopagebreak
  \textsc{Institute for Problems of Information Transmission, Moscow}\par\nopagebreak
  \textsc{Laboratory of Representation Theory and Mathematical Physics,
HSE, Moscow}\par\nopagebreak
  \textit{E-mail address:} \texttt{okounkov@math.columbia.edu}

  \medskip

  R.~Pandharipande, \textsc{Department of Mathematics, ETH Z\"urich}\par\nopagebreak
  \textit{E-mail address:} \texttt{rahul@math.ethz.ch}

}}

\date{December 2019}

\begin{abstract} We propose an explicit formula for the
$\GW/\PT$ descendent correspondence in the stationary
case for nonsingular connected
projective 3-folds. 
The formula, written in terms of
vertex operators, is found by studying the 1-leg
geometry. We prove the proposal for all nonsingular
projective toric 3-folds. An application to
the Virasoro constraints for the stationary
descendent theory of  stable pairs will
appear in a sequel.
\end{abstract}

\maketitle
\setcounter{tocdepth}{1}
\tableofcontents
\setcounter{section}{-1}
\section{Introduction}
\subsection{Correspondences} \label{vvtt1}
Let $X$ be a nonsingular projective 3-fold.
In  \cite{MNOP1,MNOP2,MOOP},
the correspondence between  {\em primary} Gromov-Witten and Donaldson-Thomas
 invariants was established in the toric case.
As indicated  in \cite{MNOP2}, the natural next step is  
to extend the map/sheaf correspondence to the full {\em descendent} theories of $X$.
A basic compatibility for the map/sheaf correspondence is that the
$\mathrm{SL}(2)$-equivariant counts in geometries of the form
$$
X = \mathsf{Curve} \times \CC^2
$$
should specialize, via Mumford's relation for Hodge classes and the analogous vanishing
on the sheaf side, to the Gromov-Witten/Hurwitz correspondence for curves studied in \cite{OPH/GW,OPP1,OPVir}. 
The idea of using the above compatibility to study the
correspondence
represents the
technical starting point of the paper, and our formulas evolved from the formulas of  \cite{OPVir}.
%\footnote{More
%precise relation between our methods and completed cycles of \cite{OPH/GW} is discussed in the proof of Proposition~\ref{prp:evPT}.} 

Our goal here is to a present a conjecture relating descendent integrals
over the moduli of stable maps and sheaves.
The conjecture is an explicit closed
formula for the correspondence for all descendents of cohomology{\footnote{We take singular cohomology always with $\mathbb{C}$-coefficients.}}  
classes
\begin{equation}\label{sstt}
\gamma\in H^{\geq 2}(X)
\end{equation}
of degree
{\em at least 2}.
We refer to the degree restriction \eqref{sstt} as the {\em stationary}
case.{\footnote{The terminology agrees with the
definition of stationary descendents in case $X$ is a curve \cite{OPH/GW}.}}
%As evidence, we prove a torus equivariant lift of the
%conjecture for 3-folds $X$ arising as 
%$\proj^1$-fibrations over surfaces with trivial canonical bundle.

Basic results on the map/sheaf descendent correspondence have been
obtained in \cite{PPDC,PPQ} including general constructions, 
proofs in the toric and hypersurface cases, calculations of
leading terms, and geometric applications. However, a closed formula
for the descendent correspondence was {\em not} found in \cite{PPDC}.
%Conjecturing such a formula for the descendent correspondence
%is our goal here. 
We have succeeded here in finding such a formula for
 descendents of classes  of degree
at least 2. For nonsingular projective toric
3-folds, we prove the stationary descendent correspondence formula (via
the methods of \cite{PPDC}).

\subsection{Stable pairs}
In the years after the first papers,
a better moduli space for the sheaf theory was introduced in \cite{pt}: the 
moduli of
stable pairs on $X$.
As the generating series of descendent integrals in the theory of stable
pairs have much better analytic properties, we will work will the
stable pairs theory on the sheaf side (instead of the moduli of
ideal sheaves used in \cite{MNOP1,MNOP2,MOOP}).

\begin{definition} \label{sp}
A {\em{stable pair}} $(F,s)$ on a 3-fold $X$ is a coherent sheaf $F$ on $X$
 and a  section $s\in H^0(X,F)$ satisfying the following stability conditions:
\begin{itemize}
\item $F$ is \emph{pure} of dimension 1,
\item the section $s:\OO_X\to F$ has cokernel of dimensional 0.
\end{itemize}
\end{definition}

Let $C$ be the scheme-theoretic support of $F$.
By the purity condition, all the irreducible components
of $C$ are of dimension 1 (no 0-dimensional components are permitted).
By  \cite[Lemma 1.6]{pt}, the kernel of $s$ is the ideal sheaf of $C$,
$$\II_C=\text{ker}(s) \subset \OO_X\, ,$$
and $C$ has no embedded points.
A stable pair $$\OO_X\to F$$ 
therefore
defines
 a Cohen-Macaulay subcurve $C\subset X$ via the kernel of $s$
 and a 0-dimensional subscheme
%\footnote{When $C$ is Gorenstein (for instance if $C$ lies
%in a nonsingular surface), stable pairs supported on $C$ are in bijection with 0-dimensional subschemes of $C$. More precis%e scheme theoretic isomorphisms of moduli spaces are proved in \cite[Appendix B]{pt3}.}
of $C$ via the support of the
 cokernel of $s$.

To a stable pair, we associate the Euler characteristic and
the class of the support $C$ of the sheaf $F$,
$$\chi(F)=n\in \mathbb{Z} \  \ \ \text{and} \ \ \ [C]=\beta\in H_2(X,\mathbb{Z})\,.$$
For fixed $n$ and $\beta$,
there is a projective moduli space of stable pairs $P_n(X,\beta)$. 
Unless $\beta$ is an effective curve class, the moduli space
$P_n(X,\beta)$ is empty. An analysis of the
deformation theory and the construction of the virtual
cycle $[P_n(X,\beta)]^{vir}$ is given \cite{pt}.
We refer the reader to \cite{P,rp13}
for an introduction to the theory of stable pairs.

\subsection{Stable pairs descendents}\label{spd}
Stable pairs invariants are integrals of the form
$$\big\langle \omega\big\rangle_{\beta}^\PT=\sum_nq^n\int_{[P_n(X,\beta)]^{vir}}\omega_n\, ,$$
where $\omega=\sum_{n\in \mathbb{Z}}\omega_n$ is an element of the formal sum $\bigoplus_{n\in \mathbb{Z}} H^*(P_n(X,\beta))$. For fixed $\beta$, the moduli
space $P_n(X,\beta)$ is empty for all sufficiently negative $n$. Hence,
$\big\langle \omega\big\rangle_{\beta}^\PT$ is a Laurent series in $q$.

Tautological descendent classes are defined via universal structures over the
moduli space of stable pairs.
Let
$$\pi: X\times P_n(X,\beta) \rightarrow P_n(X,\beta)$$
be the projection to the second factor, and let
\[\mathcal{O}_{X\times P_n(X,\beta)}\rightarrow \mathbb{F}_n\]
be the universal stable pair on $X\times P_n(X,\beta)$. Let{\footnote{We will
always take singular cohomology with $\mathbb{Q}$-coefficients.}}
\begin{equation*}
\ch_k(\mathbb{F}_n-\mathcal{O}_{X\times P_n(X,\beta)})\in  H^*(X\times P_n(X,\beta))\, .
\end{equation*}
The following
{\em descendent classes} are our main objects of study:
$$\ch_k(\gamma) = \pi_*\left(\ch_k(\mathbb{F}_n-\mathcal{O}_{X\times P_n(X,\beta)})\cdot \gamma\right)
\in H^*(P_n(X,\beta))\ \ \ \text{for}\ \ \  \gamma\in H^*(X)\, .$$
The summand $-\mathcal{O}_{X\times P_n(X,\beta)}$  only
affects $\ch_0$,
$$\ch_0(\gamma) = -\int_X \gamma \, \in \, H^0(P_n(X,\beta))\, .$$
Since stable pairs are supported on curves, the vanishing
$$\ch_1(\gamma)=0$$
always holds. 

We will study the 
following descenent series:
\begin{equation}\label{ss33}
 \Big\langle {\text{ch}}_{k_1}(\gamma_1)\cdots \text{ch}_{k_m}(\gamma_m)\Big\rangle_{\beta}^{X,\PT}
\, =\, \sum_{n\in \mathbb{Z}}  q^n\, 
\int_{[P_n(X,\beta)]^{vir}}\prod_{i=1}^m\ch_{k_i}(\gamma_i)\, .
\end{equation}
For fixed curve class $\beta\in H_2(X,\mathbb{Z})$, the moduli space
$P_n(X,\beta)$ is empty for all sufficiently negative $n$. Therefore, the 
descendent series \eqref{ss33} has only finitely many polar terms.

\begin{conjecture}\cite{pt}  \label{ggtthh}
The stable pairs descendent series 
$$ \Big\langle {\text{\em ch}}_{k_1}(\gamma_1)\cdots \text{\em ch}_{k_m}(\gamma_m)\Big\rangle_{\beta}^{X,\PT}$$ 
is the Laurent expansion of a  rational function of $q$ for
all $\gamma_i\in H^*(X)$ and all $k_i\geq 0$.
\end{conjecture}

For Calabi-Yau 3-folds, Conjecture \ref{ggtthh} reduces
immediately to the rationality of the basic series
$ \langle\, 1\, \rangle_{\beta}^\PT$
proven via wall-crossing in \cite{Br,Toda}. In the presence of
descendent insertions, Conjecture \ref{ggtthh} has been proven for
 rich class of varieties 
 \cite{part1,PPstat,PP2,PPDC, PPQ}
including all nonsingular projective toric $3$-folds.

The generating series for 
descendents in the $\mathsf{DT}$ theory of ideal sheaves
have more complicated analytic  properties. In particular,
the descendent series are {\em not} always  Laurent expansions of rational functions. 
Descendents in $\mathsf{DT}$ theory are discussed in Section \ref{sec:concluding-remarks},
and a \(\mathsf{DT}\) version of Conjecture \ref{ggtthh} is presented there.

\subsection{Gromov-Witten descendents}
Let $X$ be a nonsingular projective 3-fold.
Gromov-Witten theory is defined via integration over the moduli
space of stable maps.

Let $C$ be a possibly disconnected curve with at worst nodal singularities.
The genus of $C$ is defined by $1-\chi(\OO_C)$. 
Let $\overline{M}'_{g,m}(X,\beta)$ denote the moduli space of stable maps
with possibly {disconnected} domain
curves $C$ of genus $g$ with {\em no} collapsed connected components of genus greater or equal \(2\).
The latter condition requires 
 each non-rational and non-elliptic connected component of $C$ to represent
 a nonzero class in $H_2(X,{\mathbb Z})$.
 %In particular, 
%$C$ must represent a {nonzero} class $\beta$.
Let 
$$\text{ev}_i: \overline{M}'_{g,m}(X,\beta) \rightarrow X\, ,$$
$$ \LL_i \rightarrow \overline{M}'_{g,m}(X,\beta)$$
denote the evaluation maps and the cotangent line bundles associated to
the marked points.
Let $\gamma_1, \ldots, \gamma_m\in H^*(X)$, and
let $$\psi_i = c_1(\LL_i) \in H^2(\overline{M}'_{g,m}(X,\beta))\, .$$
The {\em descendent insertions}, denoted by $\tau_k(\gamma)$, correspond 
to the classes $\psi_i^k \text{ev}_i^*(\gamma)$ on the moduli space
of stable maps. 
Let
$$\Big\langle \tau_{k_1}(\gamma_{1}) \cdots
\tau_{k_m}(\gamma_{m})\Big\rangle^{\mathsf{GW}}_{g,\beta} = \int_{[\overline{M}'_{g,m}(X,\beta)]^{vir}} 
\prod_{i=1}^m \psi_i^{k_i} \text{ev}_i^*(\gamma_{_i})$$
denote the descendent
Gromov-Witten invariants.
The associated generating series is defined 
by
%{\footnote{Our 
%notation for $\overline{M}'_{g,r}(X,\beta)$ and $\langle,\rangle'$
%follows \cite{MNOP2}.
%However, the partition function in \cite{MNOP2} is
%denoted $\mathsf{Z}'$ (to emphasize the lack
%collapsed contributions).
%For more efficient notation, we simply use $\mathsf{Z}$.
%The degree 0 collapsed contributions
%will not appear anywhere in our paper.}} 
\begin{equation}
\label{abc}
\Big\langle   \tau_{k_1}(\gamma_{1}) \cdots
\tau_{k_m}(\gamma_{m})\Big\rangle^{\mathsf{GW}}_{\beta}= 
\sum_{g\in{\mathbb Z}} \Big \langle \prod_{i=1}^m
\tau_{k_i}(\gamma_{i}) \Big \rangle^{\mathsf{GW}}_{g,\beta} \ u^{2g-2}.
\end{equation}
Since non-rational and non-elliptic components of the domain  map nontrivially, an elementary
argument shows the genus $g$ in the  sum \eqref{abc} is bounded from below. 
 Foundational aspects of the theory
are treated, for example, in \cite{BehFan,FPn, LiTian}.

\subsection{Negative descendants}
\label{negd}
To state our 
$\GW/\PT$ descendent
conjecture, we will require 
not only the usual Gromov-Witten descendant  $\tau_k$ for $k\ge 0$ but 
also descendants $\tau_k$ with negative $k<0$ indices. 
While the negative subscripts have no geometric meaning for stable
maps, negative
descendents will drastically 
simplify the statement of the correspondence.

Negative descendents reflect the fact that descendent integrals can be interpreted as matrix coefficients
of operators in a Fock space. The Fock space formalism  for  the study of ancestors in the
\(\GW\) theory of toric manifolds was developed by Givental \cite{Giv}, and
his computations can be interpreted in terms of negative descendents.
For another application of the negative descendents, see the undergraduate thesis of Pixton \cite{PixSenior}.

We introduce the negative descendants by means of
an auxiliary algebra $\mathsf{Heis}_X$ with a  linear 
functional which encodes the  Gromov-Witten invariants.

\begin{definition}
$\mathsf{Heis}_X$ is the $\CC(u)$-algebra generated by the elements 
$$\big\{\, \tau_k(\gamma)\ \big| \  k\in\mathbb{Z}\,,  \ \gamma\in H^*(X)\, \big\} $$
and satisfying the relations 
$$ [\tau_k(\alpha),\tau_m(\beta)]=(-1)^k\frac{\delta_{k+m+1}}{u^2}\int_X \alpha\cdot\beta\, .$$
\end{definition}

The standard (shifted) Heisenberg algebra $\mathsf{Heis}$ 
is generated by $\{\tau_k\}_{k\in\mathbb{Z}}$
 with relations
$$ [\tau_k,\tau_m]=(-1)^k\frac{\delta_{k+m+1}}{u^2}\, .$$
Normally ordered monomials
$$ \tau_{i_1}\tau_{i_2}\dots\tau_{i_k},\quad i_1\le i_2\le\dots\le i_k,$$
form a linear basis of $\mathsf{Heis}$.
To an element of $\mathsf{Heis}$,
we assign an element of $\text{Hom}(H^*(X),\mathsf{Heis}_X)$ by  the following
rule on basis elements (with linear extension):
\begin{itemize}
\item[$\bullet$]
  every normally ordered monomial of positive
 degree{\footnote{Here, $\tau_{i_1}\tau_{i_2}\dots\tau_{i_k}$ has degree $k$.}}
 is assigned the $\CC$-linear map
$$\tau_{i_1}\tau_{i_2}\dots\tau_{i_k}: H^*(X) \rightarrow \mathsf{Heis}_X$$ 
defined via coproduct (we use the Swidler coproduct convention \cite{Kas}):
\begin{equation*} 
\tau_{i_1}\tau_{i_2}\cdots\tau_{i_k}(\gamma)=
\tau_{i_1}(\gamma_{(1)}) \tau_{i_2}(\gamma_{(2)})
\cdots\tau_{i_k}(\gamma_{(k)})\, .
\end{equation*}
\item[$\bullet$] the degree 0 monomial $1\in \mathsf{Heis}$ 
is assigned to $0\in \text{Hom}(H^*(X),\mathsf{Heis}_X)$.
\end{itemize}
Furthermore, for $\alpha\in H^*(X)$,
 define the product
$$\alpha\cdot\tau_{i_1}\tau_{i_2}\cdots
\tau_{i_k}(\gamma)=\tau_{i_1}\tau_{i_2}\cdots\tau_{i_k}(\alpha\cdot\gamma)\, .$$

We construct a linear functional $\langle \cdot\rangle_\beta$ on 
$\mathsf{Heis}_X$ via Gromov-Witten theory.
The positive elements $\tau_{k\ge 0}(\gamma)$  generate a commutative{\footnote{If $X$
has odd cohomology, then supercommutative.
For simplicity, our analysis will restricted to commutative case.
The modifications for odd cohomology are not significant and are
left to the reader.}} subalgebra $\mathsf{Heis}^+_X \subset \mathsf{Heis}_X$. The linear functional 
\begin{equation*}\big \langle \tau_{i_1}(\gamma_1)\tau_{i_2}(\gamma_2)\cdots
\tau_{i_k}(\gamma_k) \big\rangle_{\beta}=\big\langle 
\tau_{i_1}(\gamma_1)\tau_{i_2}(\gamma_2)\cdots \tau_{i_k}(\gamma_k)
\big\rangle_{\beta}^{\mathsf{GW}}
\end{equation*}
is well-defined on the basis elements of $\mathsf{Heis}^+_X$.
We extend the linear functional  to the whole algebra 
$\mathsf{Heis}_X$ by imposing the condition
\begin{equation}\label{ActionVacuum}
\big \langle \tau_{k}(\gamma) \Phi \big\rangle_\beta=
\big\langle \Phi\big\rangle_\beta\cdot \frac{\delta_{k+2}}{u^2}\int_X\gamma
\end{equation}
for all $\Phi\in \mathsf{Heis}_X$ and $k<0$.
We will often denote $\langle \cdot\rangle_\beta$ on $\mathsf{Heis}_X$ by 
$\langle \cdot\rangle_\beta^\GW$ to emphasis the Gromov-Witten origins.

\subsection{Renormalized  descendants}\label{vvtt3}
The most convenient  way to state our conjectural $\GW/\PT$ correspondence 
is to introduce new classes
$\mathrm{H}^{\PT}_k(\gamma)$ and $\mathrm{H}^{\GW}_k(\gamma)$ for
$\gamma \in H^*(X)$. 
The required operators are introduced below.

\vspace{8pt}
 $\bullet$
The classes $\mathrm{H}^{\DT}_k(\gamma)$ 
 are linear combinations of  descendents 
for stable pairs 
defined in Section \ref{spd}. Let
$$\Hr^{\DT}_k(\gamma)= \pi_*\left( \Hr^{\DT}_k\cdot \gamma\right)\,
\in \bigoplus_{n\in \mathbb{Z}} H^*(P_n(X,\beta))\, .$$
The classes $\Hr^{\DT}_k\in \bigoplus_{n\in \mathbb{Z}}
H^*(X\times P_n(X,\beta))$
are defined by
\begin{gather*}
\mathrm{H}^{\DT}(x)=\sum_{k=-1}^\infty x^{k+1} \Hr_k^\DT
=\mathcal{S}^{-1}\left(\frac{x}{\theta}\right)
\sum_{k=0}^\infty x^k \text{ch}_k(\mathbb{F}-\mathcal{O})\, ,
\end{gather*}
where % \(\theta^{-2}=c_2\), \(c_2\) is the second Chern class of \(T^*X\)
%$r_j$ and $c_j$ are the 
%Chern roots and the Chern classes of $T^*X$ respectively{\footnote
%{We have $c_j=e_j(r_1,r_2,r_3)$
%where $e_i$ is the elementary symmetric polynomial of degree $i$ in three variables.}}
%and
$$\theta^{-2}=-c_2(T_X),\quad \mathcal{S}(x)=\frac{e^{x/2}-e^{-x/2}}{x}\, .$$
In particular, we have
\[\mathrm{H}^{\DT}_k=\text{ch}_{k+1}(\mathbb{F})+\frac{c_2}{24}\text{ch}_{k-1}(\mathbb{F})+\frac{7c_2^2}{5760}\text{ch}_{k-3}(\mathbb{F})+\dots\]
  %-\frac{c_2^3}{322560}\text{ch}_{k-5}(\mathbb{F})+\dots\]
\vspace{10pt}
$\bullet$ The classes $\mathrm{H}^{\GW}_k(\gamma)$ are most naturally
constructed in terms
of linear combinations of descendent operators introduced by Getzler
\cite{Getzler} . These operators
are 
\begin{equation}\label{tyty}
\sum_{n=-\infty}^\infty z^n\tau_n=\mathsf{Z}^0+\sum_{n> 0}\frac{(iuz)^{n-1}}{(1+zc_1)_n}\mathfrak{a}_n+\frac{1}{c_1}\sum_{n<0}\frac{(iuz)^{n-1}}{(1+zc_1)_n}
\mathfrak{a}_n\, ,
\end{equation}
\begin{equation*}
  \mathsf{Z}^0=\frac{z^{-2}u^{-2}}{\mathcal{S}\left(\frac{zu}{\theta}\right)}
      -z^{-2}u^{-2},  
\end{equation*}
%\begin{equation*}
%\mathsf{Z}^0=\frac{\theta^{-1}z^{-1}u^{-1}}{e^{zu/2\theta}-e^{-zu/2\theta}}-z^{%-2}u^{-2},  
%\end{equation*}
where we use the standard Pochhammer symbol
$$(a)_n=\frac{\Gamma(a+n)}{\Gamma(a)}\, .$$
Here, $c_1$ is treated as a formal symbol, but
whenever there is an evaluation,  \(c_1\)
  becomes \(c_1(T_X)\).

A straightforward computation shows that the relations for the operators $\tau_k$ imply the standard Heisenberg
relations for the operators $\alpha_k$:
$$ [\mathfrak{a}_k(\alpha),\mathfrak{a}_m(\beta)]=k\delta_{k+m}\int_X\alpha\cdot \beta\, .$$
By definition \eqref{tyty},  $\tau_{k\geq 0}(\gamma)$ is a linear 
combination  
of $\mathfrak{a}_{i}(\gamma\cdot c_1^{k+1-i})$,
$i=k+1,\dots,1$,
\[\tau_k=\frac{(iu)^k}{(k+1)!}\mathfrak{a}_{k+1}-c_1\frac{(iu)^{k-1}}{k!}\left(\sum_{a=1}^k\frac1a\right)\mathfrak{a}_k
  +\frac{(iu)^{k-2}}{(k-1)!}c^2_1\left(\sum_{a=1}^{k-1}\frac{1}{a^2}+\sum_{1\le a<b\le k-1}\frac{1}{ab}\right)\mathfrak{a}_{k-1}+\ldots\, .\] 
For the first non-negative values of \(k\), the formula yields
\begin{equation}\label{ff556}
\tau_0=\mathfrak{a}_1+c_2/24\, ,
\quad \tau_1=\frac{iu}{2}\mathfrak{a}_2-c_1\mathfrak{a}_1\, ,
\quad \tau_2=-\frac{u^2}{6}\mathfrak{a}_3-\frac{3i u c_1}{4}\mathfrak{a}_2+c^2_1\mathfrak{a}_1+u^2c_2^2/5760\, .
\end{equation}

Similarly,  $\tau_{k<0}(\gamma)$ is
a linear combination of $\mathfrak{a}_{i}(\gamma\cdot c_1^{k-i})$,
$i=k,k-1,\dots$,
\[\tau_k=(-iu)^{k-1}(-k-1)!\mathfrak{a}_{k}-(-iu)^{k-2}(-k)!c_1\left(\sum_{a=1}^{-k}\frac{1}{a}\right)\mathfrak{a}_{k-1}+\ldots\, .\]
If we invert the transition matrix from elements $\mathfrak{a}$
to $\tau$, we obtain 
$$\mathfrak{a}_{-2}=-(iu)^3\tau_{-2}+\dots,\quad \mathfrak{a}_{-1}=(iu)^2\tau_{-1}
+(iu)^2c_1\tau_{-2}+\dots\, .$$
Thus the negative operators $\mathfrak{a}_{k<0}$ act  inside the bracket
in
 a  nonstandard manner:
$$\big\langle\mathfrak{a}_{k}(\gamma)\Phi\big\rangle_\beta=\left[\int_X \big(-c_1\delta_{k+1}+\delta_{k+2}iu\big)\cdot \gamma\right]\,  
\big\langle \Phi\big\rangle_\beta\, ,\quad k<0\, .$$
We assemble the operators $\mathfrak{a}$ in the following generating function:
\begin{equation}\label{eq:phiz}
\phi(z)=\sum_{n>0}\frac{\mathfrak{a}_n}{n} \left(\frac{izc_1}{u}\right)^{-n}+\frac{1}{c_1}\sum_{n<0}\frac{\mathfrak{a}_n}{n} \left(\frac{izc_1}{u}\right)^{-n}\, .
\end{equation}
% where $c_1$ is treated as a formal symbol.

% and \(\beta\) is the formal square root of the inverse of the second Chern class 
% of $T^*X$:
%%%%%$$\beta^2=-c_2.$$

% To state a formula for $\mathrm{H}^{\mathsf{GW}}_k(\gamma)$,

The main objects of our paper are the  new operators
\begin{equation}\label{eq:H^GW}
\mathrm{H}^{\mathsf{GW}}( x)=\sum_{k=-1}^{\infty} \mathrm{H}^{\mathsf {GW}}_k x^{k+1}=\frac{x}{\theta}\mathrm{Res}_{w=\infty}\left(\frac{\sqrt{dydw}}{y-w}
:e^{\theta\phi(y)-\theta\phi(w)}:\right)
\end{equation}
where $y$, $w$, and $x$ satisfy the constraint 
\begin{equation}\label{eq:master_curve}
ye^y=we^w e^{-x/\theta}\, .
\end{equation}
Here, \(\mathrm{Res}_{w=\infty}\) denotes the integral along a small loop around 
$
w=\infty. 
$
%Here, $\theta$ is a formal square root of a characteristic class:
%$$\theta^{-2}=-c_2(T^*X)\, .$$
The  operators \(\Hr^{\mathsf{GW}}_k\) are mutually commutative.
To obtain explicit formulas for \(\Hr^{\mathsf{GW}}_k\), we use the Lambert function to solve equation
(\ref{eq:master_curve}) and express \(y\) in terms of \(x,w\). 
Then, the integral in the definition
of \(\Hr^{\mathsf{GW}}_k\) can be interpreted as an 
extraction of the coefficient in of \(w^{-1}\).
We provide an explicit method to compute \(\Hr^{\mathsf{GW}}_k\) in Section~\ref{sec:labmert-function}.
The descendent classes
$$\Hr^{\mathsf{GW}}_k(\gamma) \in \mathsf{Heis}_X$$
are then obtained using the Swidler coproduct
as in Section \ref{negd}. We also use Swidler coproduct conventions in
$$\Hr^{\mathsf{GW}}_{\vec{k}}(\gamma)=\prod_{i}\Hr^{\mathsf{GW}}_{k_i}(\gamma), \quad \vec{k}=(k_1,\dots,k_m).$$

\subsection{Equivariant correspondence} \label{ccppvv} All the definitions and
construction introduced in Sections \ref{vvtt1}-\ref{vvtt3} 
have canonical lifts to the equivariant setting with respect to a group
action on the variety $X$. 
 Our first result concerns  the
equivariant  $\GW/\PT$ descendent correspondence \cite{PPDC}. 

The most natural setting is the capped vertex
formalism from \cite{MOOP,PPDC} which we review briefly here.
Let the 
3-dimensional torus 
$$\mathsf{T}=\CC^* \times \CC^* \times \CC^*$$ 
act on $\mathbf{P}^1\times\mathbf{P}^1\times\mathbf{P}^1$ diagonally
The tangent weights of the $\mathsf{T}$-action at the point 
$$\mathsf{p}=0\times 0\times 0 \in \mathbf{P}^1\times\mathbf{P}^1\times\mathbf{P}^1$$ are $s_1,s_2,s_3$. The $\mathsf{T}$-equivariant cohomology ring of a point
is 
$$H_{\mathsf{T}}(\bullet)=\CC[s_1,s_2,s_3]\, .$$
We have the following factorization 
of the restriction of class 
$c_1c_2-c_3$ of $X$ to 
 $\mathsf{p}$,
$$ c_1c_2-c_3=(s_1+s_2)(s_1+s_3)(s_2+s_3)\, ,$$
where $c_i=c_i(T_X)$.

Let $U\subset\mathbf{P}^1\times\mathbf{P}^1\times\mathbf{P}^1$ be the 
$\mathsf{T}$-equivariant $3$-fold obtained by removing the 
three $\mathsf{T}$-equivariant lines $L_1,L_2,L_3$ passing through the point $\infty\times\infty\times\infty$. Let $D_i\subset U$
be the divisor with $i^{th}$ coordinate $\infty$.  
For a  triple of partitions $\mu_1,\mu_2,\mu_3$, let
\begin{equation}\label{j233}
 \Big\langle\prod_i\,  \tau_{k_i}(\mathsf{p})\, \Big|\, \mu_1,\mu_2,\mu_3\, \Big\rangle^{\mathsf{GW},\mathsf{T}}_{U,D}\, ,\quad
  \Big \langle\, \prod_i \text{ch}_{k_i}(\mathsf{p})\, \Big|\, \mu_1,\mu_2,\mu_3\, 
\Big \rangle^{\mathsf{PT},\mathsf{T}}_{U,D}
\end{equation}
denote the 
generating series of the $\mathsf{T}$-equivariant
relative Gromov-Witten and stable pairs
 invariants of the pair $$D=\cup_i D_i\subset U$$ 
with relative conditions $\mu_i$ along the divisor $D_i$.  
The stable maps spaces are taken here with {\em no collapsed connected
components of genus greater than or equal to 2}.
The series \eqref{j233}  differ from
the {\em capped descendent vertices} of \cite{MOOP,PPDC} by our slight
change in the treatment of collapsed components.

\begin{theorem}\label{thm:two_leg}   After the change of variables 
$-q=e^{iu}$,
the following correspondence between
the 2-leg capped descendent vertices holds:
 $$ \Big\langle\, \prod_i \Hr^{\mathsf{GW}}_{k_i}(\mathsf{p})\, \Big|\, \mu_1,\mu_2,\emptyset\, 
\Big\rangle^{\mathsf{GW},\mathsf{T}}_{U,D}=
q^{-|\mu_1|-|\mu_2|}\Big\langle\, \prod_i \Hr^{\PT}_{k_i}(\mathsf{p})\, \Big|\, 
\mu_1,\mu_2,\emptyset\, \Big \rangle^{\PT,\mathsf{T}}_{U,D}\mod (s_1+s_3)(s_2+s_3)\, .$$
\end{theorem}

The result of Theorem \ref{thm:two_leg} has two defects. Since the
third partition is empty, the result only covers the {\em 2-leg} case.
Moreover, the equality of the correspondence is not proven exactly,
but only mod $(s_1+s_3)(s_2+s_3)$. For the 1-leg vertex with partitions
$(\mu_1, \emptyset,\emptyset)$, 
Theorem \ref{thm:two_leg}
can be restricted in two ways to obtain the equality of the
correspondence
$$\text{mod} \ \ (s_1+s_3)(s_1+s_2)(s_2+s_3)\, .$$
The analysis of the 1-leg geometry in Section \ref{einbein}
 shows the 
relationship of the operators 
$\Hr^{\mathsf{GW}}$ and
$\Hr^{\PT}$ to the formulas of \cite{OPH/GW,OPP1,OPVir}.

%\begin{remark} Remove later! There seems to be serious obstacles to extending the theorem to the full three leg setting. The most puzzling one
%is that the capped vertex with \(\mu_1\ne \empty\) and \(\mu_2\ne \empty\) might have poles along \(s_1+s_2=0\) up to the order \(min(|\mu_1|,|\mu_2|)\).
%We can move slightly further with CY specialization though. One can probably work out GW/PT correspondence for two two legged vertex that refines slightly
%our correspondence and works modulo \((s_1+s_2)(s_2+s_3)(s_1+s_2+s_3)\). Such correspondence would have the coefficients that are expressions of the rational functions
%of \(q\) and log derivatives of the MacMahon function. It is interesting project to pursue after completing this paper.
%\end{remark}

\subsection{Non-equivariant limit} 
\label{sec:non-equiv-limit}
%In section~\ref{sec:uniq-cor} we show that the GW/PT  correspondence constructed in this paper differs from the correspondence from \cite{PP} by the terms
%that are proportional to \(c_1c_2-c_3\). 
By following the proofs of \cite{PPDC}, we derive a 
non-equivariant $\GW/\PT$ descendent correspondence for stationary insertions.
For our statements, we will follow as closely as possible
the notation of \cite{PPDC}.

Let \(\mathsf{Heis}^{c}\) be the Heisenberg algebra with
generators \(\mathfrak{a}_{k\in\ZZ\setminus \{0\}}\),
coefficients \(\CC[c_1,c_2]\), and  relations 
\[[\mathfrak{a}_k,\mathfrak{a}_m]=k\delta_{k+m}c_1c_2.\]
Let \(\mathsf{Heis}^c_+\subset \mathsf{Heis}^c\) be the
subalgebra generated by the elements \(\mathfrak{a}_{k>0}\), and  define the \(\CC[c_1,c_2]\)-linear map
\begin{equation}\label{cclin}
\mathsf{Heis}^c\rightarrow \mathsf{Heis}^c_+\, , \ \ \
\Phi\mapsto \widehat{\Phi} 
\end{equation}
by $\widehat{\mathfrak{a}}_k= \mathfrak{a}_k$ for $k>0$ and
$$\widehat{\mathfrak{a}_k \Phi}=(-c_1\delta_{k+1}+\delta_{k+2}iu)\widehat{\Phi}\,,\ \ \
\text{for}\ k<0\, .$$
When restricted to the subalgebra \(\mathsf{Heis}^{c}_+\),
the \(\CC[c_1,c_2]\)-linear map \eqref{cclin}  is an isomorphism.

For a nonsingular projective 3-fold $X$ and classes $\gamma_1, \ldots,\gamma_l\in H^*(X)$,
the hat operation make no difference inside the $\GW$ bracket,
\begin{equation}\label{fredfredfred}
  \langle \Hr_{\vec{k}}^{\GW}(\gamma)\rangle^{\GW}_\beta=\langle     \widehat{\Hr_{\vec{k}}^{\GW}(\gamma)}\rangle^{\GW}_\beta,%\quad %\Hr_{\vec{k}}^{\GW}(\gamma)=\prod_i \Hr_{k_i}^{\GW}(\gamma_i),
\end{equation}
%\begin{equation}\label{fredfredfred}
%\big\langle \Hr_{k_1}^{\GW}(\gamma_1) \cdots
%\Hr_{k_l}^\GW(\gamma_l)\big\rangle^{\GW}_\beta=
%\big\langle 
%\widehat{\Hr}_{k_1}^{\GW}(\gamma_1) \cdots
%\widehat{\Hr}_{k_l}^\GW(\gamma_l)
%\big\rangle^{\GW}_\beta\, ,
%\end{equation}
because the treatment of the negative descendents on
the left side is compatible with the treatment of
the negative descendents by the hat operation.

Let \(\vec{k}=(k_1,\ldots,k_l)\) be a vector of non-negative integers.
Following \cite{PPDC}, we define the following element of \(\mathsf{Heis}^c_+\):
\[\widetilde{\Hr}_{\vec{k}}=\frac{1}{(c_1c_2)^{l-1}}
\sum_{\text{set partitions $P$ of \{1,\dots,l\}}}(-1)^{|P|-1}(|P|-1)!\prod_{S\in P}\widehat{\Hr}^{\mathsf{GW}}_{\vec{k}_S}\, ,\]
where \(\Hr^{\mathsf{GW}}_{\vec{k}_S}=\prod_{i\in S} \Hr^{\mathsf{GW}}_{k_i}\) and
  the element \(\Hr_k^{\GW}\in \mathsf{Heis}^c\) is a linear combination of monomials of \(\mathfrak{a}_i\), the expression is given by (\ref{eq:H^GW}).
The polynomiality of \(\widetilde{\Hr}_{\vec{k}}\)  in \(c_1,c_2\) is not obvious
(and will be deduced in Section \ref{sec:uniq-cor} 
 from the results of \cite{PPDC}).

For classes \(\gamma_1,\ldots,\gamma_l\in H^*(X)\)  and a
vector \(\vec{k}=(k_1,\ldots,k_l)\) of non-negative
integers, we define
\[\overline{\Hr_{k_1}(\gamma_1)\dots\Hr_{k_l}(\gamma_l)}=\sum_{\text{set partitions $P$ of \{1,\dots,l\}}}\,
\prod_{S\in P}\widetilde{\Hr}_{\vec{k}_S}(\gamma_S)\, ,\]
where \(\gamma_S=\prod_{i\in S}\gamma_i\).

\begin{theorem}\label{thm:noneq}
 Let  \(X\) be a nonsingular projective toric 3-fold, and let 
\(\gamma_i\in H^{\geq 2}(X,\CC)\). After the change of variables
$-q=e^{iu}$, we have
 $$\Big \langle\overline{ \Hr_{k_1}(\gamma_1)\dots \Hr_{k_l}(\gamma_l)} \Big \rangle_{\beta}^{\mathsf{GW}}=
q^{-d/2}\Big \langle \Hr^{\PT}_{k_1}(\gamma_1)\dots \Hr^{\PT}_{k_l}(\gamma_l) \Big \rangle_{\beta}^{\mathsf{PT}}
\,,$$
  where \(d=\int_\beta c_1\).
\end{theorem}

The two main restrictions in Theorem \ref{thm:noneq} are that $X$ is toric and that
the classes $\gamma_i$ are of degree at least 2. We conjecture
the first restriction to be unnecessary.

\begin{conjecture}\label{cccjjj}
 Let  \(X\) be a nonsingular projective 3-fold, and let 
\(\gamma_i\in H^{\geq 2}(X,\CC)\). After the change of variables
$-q=e^{iu}$, we have
 $$\Big \langle\overline{ \Hr_{k_1}(\gamma_1)\dots \Hr_{k_l}(\gamma_l)} \Big \rangle_{\beta}^{\mathsf{GW}}=
q^{-d/2}\Big \langle \Hr^{\PT}_{k_1}(\gamma_1)\dots \Hr^{\PT}_{k_l}(\gamma_l) \Big \rangle_{\beta}^{\mathsf{PT}}
\,,$$
  where \(d=\int_\beta c_1\).
\end{conjecture}

For the precise formula for our $\GW/\PT$ correspondence, the second
restriction (to the stationary theory) is required --- the formula
is not correct for descendents of the identity class.

\subsection{Plan of the paper} 
After reviewing the {\em dressing operator} in Section \ref{dresso},
the goal of Section \ref{einbein} is 
to establish the 1-leg version of 
Theorem \ref{thm:two_leg} with
$$\mu_1=\mu_2=\emptyset$$
 modulo \(s_1+s_2\).
We derive our formula for the  1-leg $\GW/\PT$ descendent correspondence 
  by an explicit analysis of the Gromov-Witten  and 
stable pairs descendent theory (the modulo $s_1+s_2$ condition leads
to drastic simplification). The results depend crucially on
the earlier study of curves in \cite{OPH/GW,OPP1}.
We then show our correspondence matches the
correspondence of \cite{PPDC} modulo \(c_3-c_1c_2\) and use the results of \cite{PPDC} to  conclude the proof of Theorem \ref{thm:two_leg} in Section \ref{sec:uniq-cor}.

To prove  the stationary non-equivariant result of 
Theorem \ref{thm:noneq}, we must check that the non-equivariant limit 
formulation of the $\GW/\PT$ descendent correspondence of \cite{PPDC} does not
develop singularities under the specialization \(c_3=c_2c_1\).
 The matter is discussed in the Section \ref{sec:uniq-cor}.
Examples are presented in
Section \ref{sec:examples}.

We have conjectured Virasoro constraints for the 
stable pairs descendent
theory for all nonsingular projective 3-folds (the precise
formulas for $\proj^3$ appear in \cite{P}). In a sequel \cite{OOP} to
the present paper, we will apply Theorem \ref{thm:noneq}
to obtain the Virasoro constraints for stable pairs on toric 3-folds
in the stationary case from the proven Virasoro constraints
in Gromov-Witten theory.

Section \ref{sec:concluding-remarks}
contains results and conjectures concerning parallel questions
about the descendent $\mathsf{DT}$ theory of ideal sheaves.
The ${\mathsf{DT}}$ descendent series are not always
rational functions in $q$, so a discussion of the analytic properties
is necessary.

\subsection{Past and future directions} 
The main formula and the method of the paper is quite old \cite{OkTalk}. 
Since our first draft was written, many new approaches to understanding  descendent integrals on both sides
of the correspondence were developed. In particular, we now expect a geometric
path to the \(\mathsf{GW}/\mathsf{PT}\) descendent correspondence for $X$
should be possible via {\em relative} geometries $X/D$.
For relative theories {\em without} higher descendent insertions, the correspondence is very simple \cite{MNOP1}.
After moving the descendents of the classes $$\gamma\in H^{\ge 2}(X)$$ to the
relative divisor $D$, the relative 
$\GW/\PT$ descendent correspondence there implies  a descendent correspondence for $X$.

For such a path to succeed, a detailed study of the bubble over $D$ is required.
On the sheaf side, there has been very good progress in explicitly relating the relative and descendent invariants in fully equivariant K-theory, see \cite{AO,Sm}.

\subsection{Acknowledgments} 
We are very grateful to
D.~Maulik, M.~Moreira, N.~Nekrasov,\linebreak G. Oberdieck,
A.~Pixton, J.~Shen, R.~Thomas, and Q.~Yin for  
many conversations about descendents and
descendent correspondences.

A.~Ob. was partially supported by NSF CAREER grant
DMS-1352398. This paper is based upon work supported by the National Science Foundation under Grant No. 1440140, while the first two authors were in residence at the Mathematical Sciences Research Institute in Berkeley, California during the Spring semester of 2018.
A.~Ok. was partially supported by the Simons Foundation as a Simons Investigator.
A.~Ok. gratefully acknowledges funding by the Russian
Academic Excellence Project ’5-100’ and RSF grant 16-11-10160.
R.~P. was partially supported by 
 SNF-200021143\-274, SNF-200020162928, ERC-2012-AdG-320368-MCSK, SwissMAP, and
the Einstein Stiftung.

This project has received funding from the European Research
Council (ERC) under the European Union Horizon 2020 research and
innovation program (grant agreement No. 786580).

\section{Dressing operator}\label{dresso}
\subsection{Summary}
We establish here properties of
the dressing operator $W$ which intertwines 
the operators $\AAA$ and $\widetilde{\AAA}$ of \cite{OPP1}.
 These results are needed for the
proofs of Theorems \ref{thm:two_leg} and \ref{thm:noneq}
 of the introduction.

\subsection{Notation}\label{nottt}
We recall the formulas for the operators $\AAA_k$ of \cite[Section 3.2.2]{OPP1}:
\begin{eqnarray*}
\mathcal{A}\, =\, \sum_{k\in \mathbb{Z}} \mathcal{A}_k z^{k}& =&\frac{1}{u}\mathcal{S}(uz)^{tz}
\sum_{k\in \mathbb{Z}}\frac{(e^{uz/2}-e^{-uz/2})^k}{(tz+1)_k}
\mathcal{E}_k(uz)\, ,
\\  \mathcal{E}_r(z)&=&\sum_{k\in\ZZ+\frac{1}{2}} e^{z(k-r/2)} E_{k-r,k}+\frac{
\delta_{r,0}}{(e^{z/2}-e^{-z/2})}\, ,
\\  (a+1)_k&=&\begin{cases}
             (a+1)(a+2)\dots (a+k),\quad \quad k\ge 0\\
	     (a(a-1)\dots(a+k+1))^{-1},\quad k\le 0
            \end{cases}\, .
\end{eqnarray*} 
Here, $E_{ij}$ are the matrix units of the Lie 
algebra{\footnote{Every operator in Section \ref{dresso} 
is assumed to be an element of $\mathfrak{gl}(V)$ but not 
$\mathfrak{gl}(\Lambda^{\infty/2} V)$.}}
$\mathfrak{gl}(V)$ where $V$ is the infinite dimensional $\mathbb{C}$-vector space with basis
labeled by the shifted integers $\mathbb{Z}+ \frac{1}{2}$. For a more
detailed treatment,
we refer the reader to \cite[Section 2]{OPP1}.

We will study the operator $W$ which intertwines the
 operator $\mathcal{A}$ with 
$$\widetilde{\mathcal{A}}\, =\, 
\sum_{k\in \mathbb{Z}} \widetilde{\mathcal{A}}_k z^{k}\, =\, 
\frac{1}{u}\sum_{k\in \mathbb{Z}}\frac{(uz)^k}{(tz+1)_k}\alpha_k\, ,$$
the  $u$-asymptotic expansion of $\mathcal{A}$ at 
$u\sim 0$.
See \cite[Section 4.4.2]{OPVir} for further\footnote{In \cite{OPVir}, the notation \(\mathrm{A}_k=\mathcal{A}_{k+1}\) is used.} discussion.
We have used here the operators
%$$\alpha_k=\mathcal{E}_k(0),\quad k\ne 0,\quad \alpha_0:=1.$$
 $$\alpha_k=\begin{cases}
             \mathcal{E}_k(0),\quad k\ne 0\\
	      1, \quad \quad \, \, \, \, k=0
            \end{cases}\, .$$

%Before we start further discussion let us make several remarks on the general properties of the operators $\AAA_k$.
By definition, the matrix $\mathcal{A}$ can be
 written as a series in variables $u,z,t$ with coefficients
in the subalgebra of $\mathfrak{gl}(V)$ generated by the operators
$$ H=\sum_{k\in \mathbb{Z}+\frac{1}{2}} kE_{kk} \quad \text{and} \quad  S=\alpha_{-1}\, .$$
Let us denote the latter subalgebra by $\widetilde{\mathfrak{gl}}(V)$.
The algebra 
$\widetilde{\mathfrak{gl}}(V)$
has a natural basis of ordered monomials
$$\big\{\, H^a S^b\, \big| \, a,b\in \ZZ\, \big\}$$ 
with relations 
$$ SH=(H+1)S\, .$$

The coefficient in front of each monomial $H^a S^b$ in the
formula for $\mathcal{A}$ 
%and for any operator that discussed in the text 
is a Laurent polynomials of variables $z,u,t$. In other words,
$$ \mathcal{A}\in \widetilde{\mathfrak{gl}}(V)[[z^{\pm 1},u^{\pm 1},t^{\pm 1}]]\, .$$
Moreover, $\AAA$ is homogeneous of degree $-1$ if we introduce the 
 grading 
\begin{equation} \label{grading}\deg u=\deg t=-\deg z=1\, .\end{equation}

\subsection{The differential equation}
We consider first the intertwiner between 
the operators
\begin{eqnarray*}
 D&=&S^{-1}+H\, , \\
 \widetilde{D}&=&D-\frac12 \left(H\frac{1}{1-Z}+\frac{1}{(1-Z)}H\right)
\ \ \  \text{where} \  \ \ Z=\frac{tS}{u}\,,
\end{eqnarray*}
  and 
establish the following basic
properties.

\begin{lem}\label{LemDiffEqW} There is a unique solution $W$ of the linear differential equation
\begin{equation}\frac{d W}{du}=\frac{1}{t}WB \ \ \ \
\text{where} \ \ \ \  B=H^2\frac{Z^2}{(1-Z)^2}+H\frac{Z^2}{(1-Z)^3}+
\frac{2Z^3+3Z^2}{8(1-Z)^4}\, ,\label{DiffEqW}
\end{equation}
with the following properties:
\begin{enumerate}
\item[(i)] $W|_{u=0}=1$,
\item[(ii)] $W^{-1}DW=\widetilde{D}$,
\item[(iii)] $W$ is upper-triangular.
\end{enumerate}
\end{lem}

\noindent In fact, the unique solution $W$ of Lemma \ref{LemDiffEqW}
also intertwines $\AAA$ and $\widetilde{\AAA}$:

\begin{theorem}\label{ThmDressingEq} Let $W$ be the operator 
of Lemma \ref{LemDiffEqW}, then
\begin{equation}
\label{DressingEq}
W^{-1}\mathcal{A} W=\widetilde{\mathcal{A}}\, .
\end{equation}
\end{theorem}

Existence of the {\em dressing operator}  $W$ satisfying (\ref{DressingEq}) is shown in \cite{OPP1} by  slightly different methods,
but the path
via \eqref{DiffEqW}
is new (and very efficient).
% It is possible that the statements above 
%could be derived from the result of \cite{OPP1} but the proof presented  here is arguably shorter and the equation (\ref{DiffEqW}) did make an appearance in
%\cite{OPP1}

\begin{proof}[Proof of Lemma~\ref{LemDiffEqW}]
Let $W$ be a solution of the differential equation (\ref{DiffEqW}). The equation has no singularity at $u=0$, so there is a unique solution $W$
satisfying   $W|_{u=0}=1$. A direct computation yields
$$ \frac{d}{du}\left(W\widetilde{D}W^{-1}D^{-1}\right)=W\left(\frac{1}{t}
\left[B,\widetilde{D}\right]+\frac{d\widetilde{D}}{du}\right) W^{-1}D^{-1}\, .$$
Then, after a lengthy but straightforward calculation, we find 
$$\frac{1}{t}\left[B,\widetilde{D}\right]+\frac{d\widetilde{D}}{du}=0.$$
Thus we obtain the first and second properties of $W$. 
The upper-triangularity follows from the upper-triangularity of $B$.
\end{proof}

\begin{proof}[Proof of Theorem~\ref{ThmDressingEq}]
%Let us start with some preliminary remarks. All operators that we consider are homogeneous with respect to grading 
%(\ref{grading}), hence it is would be enough to prove the statement under assumption $t=1$. Also, from the construction 
%of operator $W$ one sees that it is homogeneous of degree $0$ with respect to the grading thus it satisfies equation:
%$$\frac{dW}{du}=\frac{1}{t}WB.$$

 As explained in Section \ref{nottt}, both $\AAA$ and $W$ are  sums of monomials $H^aS^b$ with coefficients
 in the ring of Laurent polynomials of variables $u,t,z$ and,
moreover, are homogeneous with respect to grading (\ref{grading}). 
Therefore, using Zariski density,
we need only prove  \eqref{DressingEq}
 at the values
$$z=m\, ,\ \ \ t=1\, , \ \ \  m\in \ZZ_{>0}\, .$$ 
We define the operators
 $$\AAA^{(m)}=\frac{u^{m+1} m^m}{m!}\AAA|_{t=1,z=m}\,, \ \quad \widetilde{\AAA}^{(m)}=\frac{u^{m+1} m^m}{m!}\widetilde{\AAA}|_{t=1,z=m}\, .$$
 By \cite[Lemma 2]{OPP1} in first case and a 
direct computation in second case, we find
 $$\AAA^{(m)}=e^{1/S}e^{uH^2/2}S^me^{-uH^2/2} e^{-1/S}\, ,\ 
\quad \widetilde{\AAA}^{(m)}=S^m e^{mu/S}\, .$$
Thus, to prove complete the proof of Theorem 
\ref{ThmDressingEq},
we need only prove the equation 
\begin{equation}\label{frrr}
W^{-1}\AAA^{(1)} W=\widetilde{\AAA}^{(1)}\, .
\end{equation}

Let us denote the operator on the LHS of equation \eqref{frrr}
by $\mathrm{O}$ and the operator on the RHS by $\widetilde{\mathrm{O}}$. 
Equation \eqref{frrr} is 
satisfied at $u=0$ since
$$\AAA^{(1)}|_{u=0}=S\, ,\ \quad \widetilde{\AAA}^{(1)}|_{u=0}=S\,,\ 
\quad W|_{u=0}=1\, .$$
Taking the $u$ derivative of $\mathrm{O}$, we find
\begin{eqnarray*}
\frac{d \mathrm{O}}{du}&=&[\mathrm{O},B]+\frac{1}{2}W^{-1}e^{1/S}e^{uH^2/2}[H^2,S]e^{-uH^2/2}e^{-1/S}W\\
&=&[\mathrm{O},B]-\frac12 \mathrm{O}-W^{-1}e^{1/S}e^{uH^2/2}HSe^{-uH^2/2}e^{-1/S}W\\
&=&[\mathrm{O},B]-\frac12 \mathrm{O}
-\widetilde{D}\mathrm{O}\, ,
\end{eqnarray*}
where  we have used the intertwining relations for $D$ and $\widetilde{D}$.
A direct (lengthy) computation yields 
$$\frac{d\widetilde{\mathrm{O}}}{du}=[\widetilde{\mathrm{O}},B]-\frac12 \widetilde{\mathrm{O}}
-\widetilde{D}\widetilde{\mathrm{O}}.$$
By the uniqueness of a solution of a linear ODEs, the proof
of \eqref{frrr} is complete.
\end{proof}
%\begin{remark}
%From the proofs one can see that the dressing operator $W$ is upper-triangular that is a linear combination 
%of monomials $H^a S^{b}$, $b\le 0$.
%\end{remark}

%\subsection{Corollaries}
%\begin{Corollary} For any $f\in \CC[[S]]$ there exists unique dressing upper-triangular operator $W_f$ such that 
%\begin{gather*}
%W_f^{-1}\AAA W_f=\tilde{\AAA},\\
%W_f^{-1}D W_f= \tilde{D}_f,\quad \tilde{D}_f:=\frac{1}{S}-\frac{t}u(H+\frac12)S+S^2f(S). 
%\end{gather*}
%\end{Corollary}
%\begin{proof}
%Indeed, multiplication by $e^{yS^k}$ on the left and by $e^{-yS^k}$ preserves it. On the other hand we have:
%$$ e^{yS^k}\tilde{D}_f e^{-yS^k}=\tilde{D}_f-kH\frac{S^{k+1}}{1-S}.$$ 
%Thus there exists $g\in \CC[[S]]$ such that 
%$$ e^{g} \tilde{D} e^{-g}=\tilde{D}_f$$
%\end{proof}

\section{1-leg correspondence}\label{einbein}
\subsection{Background} 
%The geometry of the capped vertex was review in Section \ref{ccppvv}. 
The $1$-leg geometry concerns the space
$$X=\mathbb{C}^2 \times \proj^1$$
with the action of the 3-dimensional torus
$$\mathsf{T}= (\com^*)^2 \times \com^*\, .$$
The first factor $(\com^*)^2$ acts on $\com^2$ with tangent weight $s_1$ and
$s_2$ at the origin $0\in \mathbb{C}^2$. The second factor $\com^*$ acts on 
$\proj^1$ with tangent weights $t$ and $-t$
at the respective fixed points $0,\infty \in \proj^1$.
For simplicity, we denote the two fixed points
$$0\times 0\, ,\ 0\times \infty \in \com^2 \times \proj^1$$
by $0$ and $\infty$ respectively.

There is a 2-dimensional torus ${\mathsf{T}}_0\subset \mathsf{T}$ 
which preserves the natural symplectic form $dz_1 \wedge dz_2$
on $\CC^2$. Let
$$H_{{\mathsf{T}}_0}(\bullet)=\CC[s,t]\, ,$$ 
then the restriction to the action  ${\mathsf{T}_0}$ corresponds to the specialization 
$$s_1=-s_2=s\, , \ \ \  t=t\, .$$ 
By the Mumford identity for 
the Hodge classes, the ${\mathsf{T}_0}$-equivariant Gromov-Witten
invariants of $\com^2 \times \proj^1$ are equal to 
the $\com^*$-equivariant Gromov-Witten
 invariants of $\proj^1$ up to simple factors of $s$. 
 The results of \cite{OPP1} solving
 the equivariant Gromov-Witten theory of $\proj^1$
 in terms of operators $\AAA$
are restated in Section \ref{ffbbtt} in the form we require here.
 %To simplify notation we set $s=1$, thus we identify GW theory of $X$ with GW theory 
%of $\mathbb{P}^1$. It is easy to restore powers of $s$ by computing expected cohomological
%degrees of the invariants.

In \cite{OPP1}, 
the Gromov-Witten invariants of $\proj^1$ are computed 
in terms of the Fock space
$$ \mathcal{F}=\Lambda^{\infty/2}V\, ,\quad V= z^{1/2}\CC[[z^{\pm 1}]]\,.$$
Before stating the results of \cite{OPP1},
we give a quick overview of the basics about the Fock space
and the related representation theory.

\subsection{Fock space}
The Fock space $\Lambda^{\infty/2} V$ has a natural basis of the form
$$\Lambda^{\infty/2} V=\bigoplus_S \mathbb{C} v_S,\quad v_S=z^{s_1}\wedge z^{s_2}\wedge z^{s_3}\, \ldots,$$
where $S=\{s_1>s_2>s_3>\dots\}\subset \ZZ+1/2$ is an
ordered sequence satisfying the properties
\begin{enumerate}
\item[(i)] $S_+=S\setminus 
\left\{\mathbb{Z}_{\le 0}-\frac{1}{2}\right\}$ is finite,
\item[(ii)] $S_-=\left\{\mathbb{Z}_{\le 0}-\frac{1}{2}\right\}\setminus S$ is finite.
\end{enumerate}
The fermionic operator $\psi_k$ on $\Lambda^{\infty/2}V$ is 
defined by wedge product with the vector $z^k$,
$$\psi_k\cdot v=z^k\wedge v.$$

An inner product $(\cdot,\cdot)_z$ on $V$ is
defined by letting the monomials $z^k$ be an 
orthonormal basis. We use the same notation $(\cdot,\cdot)_z$
for the induced inner product on $\Lambda^{\infty/2}V$. Let
$A^*$ denote the operator adjoint to $A$ with respect to the inner 
product $(\cdot,\cdot)_z$.  The
adjoint operators $\psi^*_k$ satisfy 
the canonical anti-commutation relations,
\begin{gather*}
\psi_i\psi^*_j+\psi_j^*\psi_i=\delta_{ij}\, ,\quad
\psi_i\psi_j+\psi_j\psi_i=\psi_i^*\psi_j^*+\psi_j^*\psi_i^*=0\, . 
\end{gather*}

The projective representation $\pi_\mathcal{F}$ of $\mathfrak{gl}(V)$ is defined in terms of fermion operators by the formula
$$ \pi_{\mathcal{F}}(E_{ij})=:\psi_i\psi_j^*:\, ,$$
where we have used normal order notation
$$ :\psi_i\psi^*_j :=\begin{cases} \psi_i\psi_j^*\, ,
 \quad j>0\\ -\psi_j^*\psi_i\, ,\quad j<0\, .\end{cases}$$
 
The operators $\alpha_k$ of Section \ref{nottt}  commute as element of 
$\mathsf{End}(V)$, but the operators $\pi_{\mathcal{F}}(\alpha_k)$ do 
{\em not} commute --- the operators $\pi_{\mathcal{F}}(\alpha_k)$
form an Heisenberg algebra. For simpler formulas, 
we will drop $\pi_{\mathcal{F}}$ in our notation.
That is, we use
$\alpha_k$ for the operators $\pi_{\mathcal{F}}(\alpha_k)$:
$$  [\alpha_k,\alpha_l]=k\delta_{k+l}\, .$$
The action of the Heisenberg algebra preserves the eigenspaces of the charge operator
$$ C v_S=(|S_+|-|S_-|) v_S\, .$$

The operators $\AAA_k$, $\widetilde{\AAA}_k$
of Section \ref{nottt} act on $\Lambda^{\infty/2}V$ via $\pi_{\mathcal{F}}$.
We obtain an alternative proof
of Theorem 1 of \cite{OPP1}.
\begin{Corollary}  \label{pqq2} As elements of \(\mathrm{End}(\Lambda^{\infty/2}V)\) the operators
  \(\widetilde{\AAA}_k\) satisfy 
$$\left[\widetilde{\AAA}_k,\widetilde{\AAA}_l\right]=(-1)^k \delta_{k+l-1}\frac{t}{u^2}\, .$$
\end{Corollary}

\begin{proof} Using the homogeneity of $\widetilde{\AAA}$, 
we set \(u=1\) for the
 proof. The statement of Corollary \ref{pqq2} is equivalent
to the equation
$$\left[\widetilde{\AAA}(z),\widetilde{\AAA}(w)\right]=tz\sum_{n\in\ZZ}\left(-\frac{z}{w}\right)^n$$
which we will derive from the Heisenberg relations for the operators
$\alpha_k$. By definition (after setting $u=1$),
$$\left[\widetilde{\AAA}(z),\widetilde{\AAA}(w)\right]=
%\frac{1}{u^2}
\sum_{n\ne 0}n\left(\frac{z}{w}\right)^{n}\frac{1}{(1+tz)_n(1+tw)_{-n}}\, .$$
On other hand, the summation over  positive $n$ after
multiplication by 
$(1+\frac{z}{w})$ is equal to
$zt$ because
\begin{multline*}
\left(\frac{z}{w}\right)^n\left(\frac{n}{(1+tz)_n(1+tw)_{-n}}+\frac{z}{w}\frac{n}{(1+tz)_n(1+tw)_{-n}}\right)=\\
\left(\frac{z}{w}\right)^n\left(\frac{1}{(1+tz)_{n-1}(1+tw)_{-n}}-\frac{tz}{(1+tz)_n(1+tw)_{-n}}\right)-\\
\left(\frac{z}{w}\right)^n\left(\frac{z}{w}\frac{1}{(1+tz)_{n}(1+tw)_{-n-1}}-\frac{tz}{(1+tz)_n(1+tw)_{-n}}\right)=\\
\left(\frac{z}{w}\right)^n\frac{1}{(1+tz)_{n-1}(1+tw)_{-n}}-\left(\frac{z}{w}\right)^{n+1}\frac{1}{(1+tz)_n(1+tw)_{-n-1}}\, .
\end{multline*}
Analogously, the summation over
  negative $n$ after multiplication by $1+\frac{w}{z}$ is equal $-zt$.
\end{proof}

The Fock space contains a special {\em vacuum vector}  
$$v_\emptyset=v_{-1/2\, ,\,-3/2\, ,\, -5/2\, ,\, \ldots}\, $$
annihilated by the positive part of the Heisenberg algebra spanned by 
$\alpha_{k>0}$. 
The vectors 
$$ |\mu\rangle =\frac{1}{\mathfrak{z}(\nu)} \prod \alpha_{-\nu_i} v_\emptyset$$
form a basis of the Fock subspace of vectors of charge $0$.
For an operator $A\in \mathsf{End}(\mathcal{F})$,
 the shorthand notation 
$$\langle A|\mu\rangle^{\mathcal{F}} =\left(v_\emptyset,A|\mu\rangle\right)_z$$
is commonly used.

The Gromov-Witten bracket of \cite{OPP1}, 
$$\big\langle \tau_{k_1}([0])\tau_{k_2}([0])\dots\tau_{k_n}([0]) 
\,\big|\, \mu\big\rangle^{\mathsf{GW},\CC^*,\bullet}_{\proj^1}\, ,$$
denotes the $\CC^*$-equivariant theory of $\proj^1$ relative to 
$\infty\in \proj^1$.
The superscript $\bullet$
indicates integrals over the moduli spaces of
stable relative maps
with possibly disconnected domains. In particular,
we allow collapsed connected components of all genera
(but stability of the map must be respected).
A central result of \cite{OPP1} is the following matching.

\begin{theorem}\cite{OPP1} The
$\CC^*$-equivariant Gromov-Witten
theory of $\proj^1$ satisfies
$$\big \langle \tau_{k_1}([0])\tau_{k_2}([0])\dots\tau_{k_n}([0]) 
\, \big|\, \mu\big\rangle^{\mathsf{GW},\CC^*,\bullet}_{\proj^1}=
\big \langle \AAA_{k_1+1}\AAA_{k_2+1}\dots\AAA_{k_n+1} e^{\alpha_1}\, \big|\, \mu
\big\rangle^{\mathcal{F}}\, .$$ \label{jjjjj}
\end{theorem}

\subsection{$\GW$  in terms of the Fock space}\label{ffbbtt}
To state the analogous formula for invariants of 
$$X=\CC^2 \times \proj^1$$ with
descendents  placed at the fixed point $[0]\in \CC^2\times\proj^1$,
a slight modification
$\bar{\AAA}_k$
of operator $\AAA_k$ is required. Let
$$ \bar{\AAA}_k=s^2\Psi(\AAA_k)%+\delta_{k+1}/u^2
\, ,$$
where $\Psi$ is the following homomorphism of $\CC[t]$-algebras:
$$\Psi:\widetilde{\mathfrak{gl}}(V)[[u^{\pm 1},t^{\pm 1}]]\to 
\widetilde{\mathfrak{gl}}(V)[[u^{\pm 1},t^{\pm 1},s^{\pm 1}]],
\quad \Psi(u)=ius,\quad \Psi(\alpha_k)=\alpha_k/s^k\, .$$

The above modification of operators is chosen in such way that the
identification 
$$\bar{\AAA}_{k+1}=\tau_k([0])\, ,\quad k\in \ZZ$$
defines a homomorphism from the subalgebra of 
$\widetilde{\mathfrak{gl}}(V)[[u^{\pm 1}, t^{\pm 1}, s^{\pm 1}]]$ 
generated $\AAA_k$ to the algebra $\mathsf{Heis}_X$
of Section \ref{negd}.
 Moreover, the action of $$\bar{\AAA}_k\,, \quad  k<0$$ on 
the vacuum is consistent 
with \eqref{ActionVacuum}.
After combining previous remarks with Theorem \ref{jjjjj}, 
we obtain the following formula for 
invariants with the relative condition $\mu_1([\infty])\dots\mu_m([\infty])$ 
over $\infty\in \proj^1$.

\begin{proposition}\label{prp:evGW}
The
$\mathsf{T}_0$-equivariant Gromov-Witten
theory of $X$ satisfies
$$\big \langle \tau_{k_1}([0])\tau_{k_2}([0])\dots\tau_{k_n}([0]) 
\, |\, \mu\big \rangle^{\mathsf{GW},\mathsf{T}_0,\bullet}_{X}=\big \langle \bar{\AAA}_{k_1+1}\bar{\AAA}_{k_2+1}\dots\bar{\AAA}_{k_n+1} e^{\alpha_1}\, |\, \mu
\big \rangle^{\mathcal{F}}\,. $$
\end{proposition}

%For our convenience let us also fix an isomorphism between the algebra 
%$Heis$ and $Heis_{0}$:
%$$\alpha_k\mapsto \alpha_k([0])/(st), \quad \alpha_{-k}\mapsto \alpha_k([0])/s,\quad k>0.$$

\subsection{$\PT$ in terms of Fock space}
The set of half-integers indexes the standard basis
$$\{\, e_i\, \}_ {i\in \mathbb{Z}+\frac{1}{2}}$$ of \(V\). 
We identify the vector space $V$ with
\(z^{1/2}\CC[z][[z^{-1}]]\) by 
$$ e_i \mapsto z^i\, .$$
Then, the operator
$\alpha_n$ acts as
multiplication by \(z^{-n}\) 
 on
\(z^{1/2}\CC[z][[z^{-1}]]\).
We define the operator $H$ on $V$ by
 \(z\frac{d}{dz}\) on \(z^{1/2}\CC[z][[z^{-1}]]\).
On the sheaf side, the insertion of the descendants is given by the following
formula involving $H$.

\begin{proposition}\label{prp:evPT}The
$\mathsf{T}_0$-equivariant stable pairs
theory of $X$ satisfies
$$q^{-|\mu|}\big \langle \prod_{j=1}^\ell \Hr^{\DT}_{k_j}([0])\, 
\big|\, \mu\big \rangle^{\DT,\mathsf{T}_0}_X=(-s)^{\ell}[x^{\vec{k}}]
\big \langle  
\prod_{j=1}^{\ell}e^{x_isD}e^{\alpha_1}\, \big|\, \mu
\big \rangle^{\mathcal{F}},$$
where $D=H+\alpha_1$.% and $H=z\frac{d}{dz}$.
\end{proposition}
\begin{proof}
By standard arguments \cite{MNOP1,pt}, the moduli space
$P_n(X,d)$ is empty if $n<d$. On the other hand if \(n>d\) the virtual cycle on
\(P_n(X,n)\) vanishes in the \(\mathsf{T}_0\)-theory \cite{MPT}.
If $n=d$,  
$$P_n(X,n) =\mathsf{Hilb}_n(\CC^2)$$ 
is nonsingular of expected dimension, and 
$$\text{ch}(\mathcal{O}-\mathbb{F})= \text{ch}(\mathcal{I})$$
 where \(\mathcal{I}\) on $\CC^2 \times \mathsf{Hilb}_n(\CC^2)$ is the universal ideal sheaf
associated to $\mathsf{Hilb}_n(\CC^2)$,
$$\pi:\CC^2 \times \mathsf{Hilb}_n(\CC^2) \rightarrow \mathsf{Hilb}_n(\CC^2)\, .$$

On the other hand, Nakajima's construction provides a
 natural identification between the Fock space and 
$$\bigoplus_{n=0}^\infty H_{\mathsf{T}_0}(\mathsf{Hilb}_n(\CC^2))\, .$$
By equivariant Grothendieck-Riemann-Roch,
 $\pi_*\text{ch}(\mathcal{I})$ is expressible in terms of the
Chern character of the tautological
sheaf over the Hilbert scheme $\mathsf{Hilb}_n(\CC^2)$. 
The operator multiplication by the latter
 Chern character is diagonal in the basis of torus fixed points and has 
 a simple expression in terms of $H$ (see, for example, \cite{WangAndAll}).
 In the Fock space model, we have,
$$(e^{xs/2}-e^{-xs/2})^{-1}\pi_*\text{ch}(\mathcal{I})(x)=\pi_{\mathcal{F}}(e^{xsH})\, .$$
Further discussion of the properties of the operator $\pi_{\mathcal{F}}(e^{xsH})$
can be found in
\cite[Section 2.2.1]{OPH/GW}.
 In other words, we have  
$$\Big \langle \prod_{j=1}^\ell \Hr^{\DT}_{k_j}([0])\, \Big|\, 
\mu\Big\rangle^{\DT,\mathsf{T}_0}_X=q^{|\mu|}(-s)^\ell [x^{\vec{k}}]
\Big \langle  e^{\alpha_1}
\prod_{j=1}^\ell e^{x_jsH}\, \Big|\, \mu\Big\rangle^{\mathcal{F}}\, .$$
The last claim of the Proposition follows from  the formula
$e^{\alpha_1}H=(H+\alpha_1)e^{\alpha_1}$.
\end{proof}

\subsection{The dressing operator and the $\GW/\PT$ operators} 
The dressing operator $$\overline{W}=\Psi(W)$$ drastically 
simplifies the formula for the Gromov-Witten invariants of $X$. 
Indeed, by the results of Section \ref{dresso}, we have:

\begin{gather} \label{eq:dress}
\overline{W}^{-1}\left( \sum_{n\in\ZZ} x^{n} \bar{\AAA}_n\right) 
\overline{W}=\sum_{n>0} \frac{(iu)^{n-1}x^n}{(1+tx)_n}\bar{\alpha}_n+
\frac{1}{t}\sum_{n<0} \frac{(iu)^{n-1}x^n}{(1+tx)_n}\bar{\alpha}_n\, ,\\
\overline{W}^{-1}(\alpha_1+H)\overline{W}=\mathrm{D}\, ,\\
\mathrm{D}=\frac{\alpha_1}{s}-
\sum_{n>0}\left( H+n/2\right)\left(\frac{t}{iu}\right)^n\alpha_{-n}.\label{eq:opD}
\end{gather}
where we define
$$\bar{\alpha}_k=s\alpha_k\, ,\quad\ \ \  \bar{\alpha}_{-k}=st\alpha_{-k}-t\delta_{k-1}+\delta_{k-2}iu\, , \ \ \quad k>0\, .$$

Immediately from the formulas, we see that the operators
 $\bar{\alpha}_k$ satisfy the same relations as the operators 
$\mathfrak{a}_k([0])$ from $\mathsf{Heis}_X$.  
Moreover, since $\overline{W}$ is upper-triangular (and thus preserves
the vacuum), we have the following formulas for the invariants:
\begin{eqnarray}\label{eq:evGW}
\Big\langle\prod_{j=1}^\ell\mathfrak{a}_{k_j}([0])\Big|\mu\Big\rangle^{\mathsf{GW},\mathsf{T}_0,\bullet}_X&=&
\Big\langle\prod_{j=1}^\ell\bar{\alpha}_{k_j}\overline{W}^{-1}e^{\alpha_1}\Big|\mu\Big\rangle^{\mathcal{F}}\, , \\
\label{eq:evPT}
  q^{-|\mu|}\Big\langle\prod_{j=1}^\ell\Hr^{\DT}_{k_j}([0])\Big|\mu
  \Big\rangle^{\DT,\mathsf{T}_0}_X&=&(-s)^\ell[x^{\vec{k}}]
\Big\langle\prod_{j=1}^\ell e^{x_js\mathrm{D}}
\overline{W}^{-1}e^{\alpha_1}\Big|\mu\Big\rangle^{\mathcal{F}}\, .
\end{eqnarray}
%where \(d=|\mu|\).

To prove \eqref{eq:evGW}, we start from definition \eqref{tyty}.
The first formula is equivalent to the evaluation of the generating functions for \(\mathfrak{a}_k\) from the definition (\ref{tyty}):
\begin{multline*}
\Big\langle\prod_i\sum_{n_i}\tau_{n_i}([0])x_i^{n_i}\Big|\mu\Big\rangle^{\mathsf{GW},\mathsf{T}_0,\bullet}_X = \\
\Bigg\langle \prod_i\Bigg( \sum_{n_i>0} \frac{(iux_i)^{n_i-1}}{(1+tx_i)_{n_i}}\mathfrak{a}_{n_i}([0])+
    \frac{1}{t}\sum_{n_i<0} \frac{(iux_i)^{n_i-1}}{(1+tx_i)_{n_i}}\mathfrak{a}_{n_i}([0])\Bigg)\Bigg|\mu\Bigg\rangle^{\mathsf{GW},\mathsf{T}_0,\bullet}\, .
\end{multline*}
Then, using Proposition \ref{prp:evGW}, the vacuum preservation of \(\overline{W}^{-1}\),  and \eqref{eq:dress}, we find
\begin{eqnarray*}
\Big\langle\prod_i\sum_{n_i}\tau_{n_i}([0])x_i^{n_i}\Big|\mu\Big\rangle^{\mathsf{GW},\mathsf{T}_0,\bullet}_X
 &=& 
\Big\langle \prod_i \bar{\AAA}(x_i)e^{\alpha_1}\Big|\mu\Big\rangle^{\mathcal{F}}
\\ &= &
\Big\langle\overline{W}^{-1}\prod_i\bar{\AAA}(x_i) e^{\alpha_1}\Big|\mu\Big\rangle^{\mathcal{F}} \\ &=&
\Bigg\langle\prod_i\Bigg( \sum_{n_i>0} \frac{(iu x_i)^{n_i-1}}{(1+tx_i)_{n_i}}\bar{\alpha}_{n_i}+
        \frac{1}{t}\sum_{n_i<0} \frac{(iux_i)^{n_i-1}}{(1+tx_i)_{n_i}}\bar{\alpha}_{n_i}\Bigg)\overline{W}^{-1}e^{\alpha_1}\Bigg|\mu\Bigg\rangle^{\mathcal{F}}\, .
\end{eqnarray*}
Equation \eqref{eq:evGW} follows from these two equations.
The proof of \eqref{eq:evPT} is simpler (and uses Proposition \ref{prp:evPT}).

In order to approach Theorem \ref{thm:two_leg} in the 1-leg case, we will require the following
result.
 
\begin{prop}\label{BosonFermion}
The following identity holds in $\mathsf{End}(\mathcal{F})[[u^{\pm 1},t^{\pm 1}, s^{\pm 1}]]$:
$$ e^{xs\mathrm{D}}=\oint_{|y|=1/\epsilon}\frac{\sqrt{dydw}}{y-w} \exp\Big(\frac{iu}{2st}(w^{2}-y^{2})
  +\frac{1}{s}(y-w)\Big):\exp\Big( \sum_{n\ne 0} \frac{\alpha_n}{n}\big(y^{-n}-w^{-n}\big)\Big):\, ,$$
where the integral is taken on the surface defined by the equation
\[ye^{-iyt/u}=we^{-iwt/u}e^{sx}.\]
\end{prop}

In the statement of Proposition \ref{BosonFermion}, we have used normal ordering notation:
\begin{equation*}
 :\alpha_i\alpha_{-i}:
 =\begin{cases} \alpha_i\alpha_{-i}\, ,\quad i<0\\
                                    \alpha_{-i}\alpha_i\, ,\quad i>0\, .
                                  \end{cases}
\end{equation*}

\subsection{Proof of Proposition~\ref{BosonFermion}}
%The set of half-integers indexes the standard basis
%$$\{\, e_i\, \}_ {i\in \mathbb{Z}+\frac{1}{2}}$$ of \(V\). 
%We identify the vector space $V$ with
%\(z^{1/2}\CC[z][[z^{-1}]]\) by 
%$$ e_i \mapsto z^i\, .$$
We have seen that the operators
\(H\) and $\alpha_n$ act respectively
 as \(z\frac{d}{dz}\) and multiplication by \(z^{-n}\) 
 on
\(z^{1/2}\CC[z][[z^{-1}]]\). Thus, \(\mathrm{D}\), defined by 
equation \eqref{eq:opD}, becomes  a differential
operator acting on the functions of \(z\).

We view  $\mathrm{D}$ as acting on functions of $z$ from the left.
Consider the eigenvalue problem 
\begin{equation}\label{gg555}
s\mathrm{D}f=\lambda f,\quad f\in z^{1/2}\mathrm{Hol}(\CC^*)\, ,
\end{equation}
where $\mathrm{Hol}(\CC^*)$ denotes holomorphic single-valued functions
on $\CC^*$. The Laurent
series expansion  provides a map from \(z^{1/2}\mathrm{Hol}(\CC^*)\) to the
 completion \(\overline{V}\) of the space \(V\).

Since $\alpha_n$ acts as multiplication by $z^{-n}$, the eigenvalue
equation \eqref{gg555} is equivalent to the following 
ODE:%\footnote{Since $\frac{d}{dz}$ acts on the left on the function $f$, the
%action is the negative of the usual derivative action.}

\begin{gather*}
\left[\frac{1}{z}-\frac{d}{dz}\frac{stz^2}{iu-tz}-s
  a\left(\frac{tz}{iu}\right)   \right]f= \lambda f\, ,\quad \quad
a(x)=\frac{1}{2}\frac{x}{(1-x)^2}\, ,
\end{gather*}
with solution
$$ f=z^{1/2-\lambda/s} \exp\left[\frac{iu}{2stz^2}-\left(\frac{1}{s}+\frac{iu\lambda}{st}
\right)\frac{1}{z}\right]\left(1-\frac{tz}{iu}\right)^{-1/2}.$$
The condition $f\in z^{1/2}\CC[z][[z^{-1}]]$ leads to the eigenfunctions:
\begin{equation*}
f_k=(z^{-1}e^{-\frac{iu}{tz}})^{k+1/2}\exp\left[ -\frac{iu}{2stz^2}
+\left(\frac{1}{s}\right)\frac{1}{z}\right]\left(1-\frac{iu}{tz}\right)^{-1/2}\,,
%\label{f_k}
\quad \quad
\lambda_k=s\left(k+\frac{1}{2}\right)\, ,
%\label{lambda}
\end{equation*}
for $k\in\ZZ$.

For  $h(z)\in 1+z^{-1}\CC[[z^{-1}]]$, the operator of multiplication by $h(z)$,
$$ \mathrm{M}_{h}: z^k\mapsto z^k\cdot h(z)\, ,$$
is an invertible endomorphism of \(V\). Similarly, for $\theta(z)\in z^{-1}+z^{-2}\CC[[z^{-1}]]$, the reparametrization operator
$$ \mathrm{R}_{\theta}:  z^k\mapsto \theta(z)^k$$
is invertible. We can therefore
restate the above computation in terms of multiplication and reparametrization operators,
\begin{equation}\label{xx556}
sH=\mathrm{R}^{-1}_{\theta} \mathrm{M}_{e^\mathrm{g}}^{-1}\mathrm{D}\mathrm{M}_{e^\mathrm{g}}\mathrm{R}_\theta \, ,
\end{equation}
where 
\begin{equation*}
\mathrm{g}(z)= -\frac{iu}{2stz^2}
+\left(\frac{1}{s}\right)\frac{1}{z}-\frac12 \log\left(1-\frac{iu}{zt}\right)\, ,
\quad \quad \theta(z)=z^{-1}e^{-\frac{iu}{tz}}\, .
\end{equation*}

In the proof of Proposition  \ref{BosonFermion} so far,
we have studied operators in \(\mathrm{End}(V)\). The claim
of Proposition \ref{BosonFermion},
on the other hand, is about the operators in \(\mathrm{End({\mathcal{F}})}\).
For the remainder of the proof, we will work in
\(\mathrm{End({\mathcal{F}})}\).
 We will use formula \eqref{xx556} to
find an expression for $\mathrm{D}$ in terms of the operators $$\alpha_n\in  \mathrm{End({\mathcal{F}})} $$
via  the boson/fermion correspondence.

Let us quickly review the key points  of the boson/fermion correspondence.
 It is customary to assemble fermionic operators in generating functions:
$$\psi(x)=\sum_{k\in \ZZ+1/2} \psi_k x^k\, ,\quad
\quad \psi^*(x)=\sum_{k\in \ZZ+1/2} \psi^*_k x^{-k}\, .$$ 
The zero mode of the product of two fermionic generating function gives  the
exponential of the operator $sxH$:
\begin{equation}\label{bfbf}
 e^{sxH}=[y^0]\psi(y)\psi^*(ye^{-sx})\, .
\end{equation}
Thus, to express $e^{sx\mathrm{D}}$ in terms of the operators $\alpha_n$
using \eqref{xx556},
 we must compute the action of the reparametrization and
scaling operators on $\psi(x)$ and $\psi^*(x)$.

%An element $B\in GL(V)$ induces the linear  operator $B_y$ on the space of $GL(\Lambda^{\infty/2})$-valued
%currents:
%$$B_y(G(y))=B G(y) B^{-1}.$$
%Also operator $B$ acts on $y^{1/2}\CC[[y^{\pm1 }]]$ in the natural way and these two actions are related:
\begin{Lemma}  \label{pp23}
For  $g\in z^{-1}+z^{-2}\CC[[z^{-1}]]$, we have:
$$ \mathrm{R}_{g}\psi(x) \mathrm{R}^{-1}_{g}=-\psi\left(\frac{1}{g^{inv}(1/x)}\right)x(\log( g^{inv}(1/x)))_x\, ,$$
$$ \mathrm{R}_{g}\psi^*(x)\mathrm{R}^{-1}_g=\psi\left(\frac{1}{g^{inv}(1/x)}\right)\, ,$$
where \(f_x\) stands for the \(x\)-derivative of \(f\) and $g^{inv}$ is the inverse function 
$$g^{inv}(g(x))=x\, .$$
 \end{Lemma}
\begin{proof}
The matrix coefficients of the operator $\mathrm{R}_g$ are given by the expansion 
$$ \mathrm{R}_g(z^k)=\sum_i r_{ik} z^i.$$ 
We then have the following formulas 
with summation indices \(i,k\) ranging in the
set \(1/2+\mathbb{Z}\):
\begin{eqnarray*}
\mathrm{R}_g\psi(y) \mathrm{R}_g^{-1}(f)&=&\sum_k \mathrm{R}_g( y^kz^k\wedge \mathrm{R}_g^{-1}(f))\\
&=&\sum_k y^k \mathrm{R}_g (z^k)\wedge f\\ &=&
\sum_{k,i} r_{ik}y^k z^i\wedge g\\
&=&\sum_i \mathrm{R}_g^*(y^i) z^i\wedge g\, ,
\end{eqnarray*}
where $\mathrm{R}^*_g$ is the linear operator adjoint with respect to the scalar product $(\cdot,\cdot)_y$ with the orthonormal basis $y^i$.
To complete the proof, we must compute the adjoint operator $\mathrm{R}^*_g$:
\begin{eqnarray*}
(y^m, \mathrm{R}_{g}(y^k))_y&=&\oint y^m g^k(1/y)\frac{dy}{y}\\
&=&-\oint \left(\frac{1}{g^{inv}(1/w)}\right)^m w^{-k}\frac{g^{inv}(1/w)_w}{g^{inv}(1/w)} dw\\&=&
-\Big(w(\log(g^{inv}(1/w)))_w {\mathrm{R}}_{1/g^{inv}(1/w)}(w^m),w^k\Big)_w\, .
\end{eqnarray*}
We conclude
\[\mathrm{R}^*_{g}(y^i)=y(\log(g^{inv}(1/y))_y)\mathrm{R}_{1/g^{inv}}(y^i)\, ,\]
which implies the first equation of the Lemma.
The second equation,
$$ \mathrm{R}_g\psi^*(y)\mathrm{R}^{-1}_g=((\mathrm{R}^*_g)^{-1}(\psi(1/y)) \mathrm{R}_g^*)^*=(\psi(\mathrm{R}_g^{-1}(1/y)))^*=\psi^*\left(
\frac{1}{\mathrm{R}_{g}^{-1}(1/y)}\right)\, ,$$
then follows from the first.
\end{proof}

%Below we use the $GL(\Lambda^{\infty/2}V)$-valued scalar product $(\cdot,\cdot)_y$ on the space of 
%$GL(\Lambda^{\infty/2}V)$-valued currents. This scalar product is induced by the scalar product on $\CC[y^{\pm 1}]$
%defined by requiring that the monomials $y^i$ form an orthonormal basis. To apply Lemma we also need an inverting 
%operator $\mathrm{I}$ on $\CC[y^{\pm^1}]$ which sends $y^i$ to $y^{-i}$:

By applying Lemma \ref{pp23}, we obtain
\begin{eqnarray}\nonumber
[y^0]\mathrm{R}_\theta\psi(\xi)\psi^*(\xi e^{-sx})\mathrm{R}^{-1}_\theta
&=&-[\xi^0]\psi\left(\frac{1}{\theta^{inv}(1/\xi)}\right)\psi^*\left(\frac{1}{\theta^{inv}(e^{-sx}/\xi)}\right)
\xi (\ln(\theta^{inv}(1/\xi)))_\xi\\
\label{fred999}
&=&\oint_{|y|=1/\epsilon}\psi(y)\psi^*(w)\frac{dy}{y}\, ,
\end{eqnarray}
where $y=\frac{1}{\theta^{inv}(1/\xi)}$, $w=\frac{1}{\theta^{inv}(e^{-sx}/\xi)}$ and \(\epsilon\) is close to zero. The variables $y,w$ 
are subject to the constraint:

%(\mathrm{U}_y(\psi(y)),\mathrm{I}\circ\mathrm{U}_{ye^{-sx}}(\psi^*(ye^{-sx}))_y=\\
%\left(\psi\left[\mathrm{U}^*(y)\right],\mathrm{I}\circ\psi^*\left[
%\frac{1}{\mathrm{U}^{-1}_{ye^{-sx}}(e^{sx}/y)}\right]\right)_y=
%\left(\psi\left[y\right],\psi^*\left[
%\mathrm{U}_y\circ\mathrm{I}\circ\frac{1}{\mathrm{U}^{-1}_{ye^{-sx}}(e^{sx}/y)}\right]\right)_y=\\
%[y^0]\psi\left(y\right)\psi^*\left(\mathrm{I}\circ
%\mathrm{U}_y\circ\mathrm{I}\circ\frac{1}{\mathrm{U}^{-1}_{ye^{-sx}}(e^{sx}/y)}\right)
%\end{multline*}
%Let us denote by $w$ the argument of $\psi^*$ in the last line of the equation above.
%The construction of the map $\mathrm{U}$ implies that $w$ is a solution of implicit equation:
$$  \theta(1/w)=\theta(1/y)e^{sx}.$$

The boson/fermion correspondence is written in terms of these generating functions and  the following auxiliary operators:
\begin{gather*}
\psi(x)=Tx^{C+1/2}\Gamma_+(x)\, ,\quad \psi^*(x)=T^{-1}x^{-C+1/2}\Gamma_-(x)\, ,\\
\Gamma_{\pm}(x)=e^{\pm \alpha_-(x)}e^{\pm\alpha_+(x)}\, ,\quad \alpha_{\pm}(x)=
\mp\sum_{k>0}\alpha_{\pm k}\frac{x^{\mp k}}{k}\, ,
\end{gather*}
where \(C\) is the standard charge operator \cite{OPH/GW} and $T$ is the shift operator
$$ T(v_S)=v_{s_1+1,s_2+1,s_3+1,\dots}\, .$$

The boson/fermion correspondence \eqref{bfbf} together with 
\eqref{xx556} and \eqref{fred999} yields:
\begin{gather*} 
e^{sx\mathrm{D}}=\oint_{|y|=1/\epsilon} \frac{dy}{y}\sqrt{\frac{w}{y}} \mathrm{S} \Gamma_+(y)
\Gamma_-(w)\mathrm{S}^{-1}\, ,\quad w e^{-iuw/t}=ye^{-iuy/t}e^{-sx}\, ,
\end{gather*}
where \(\epsilon\) is very small and \(S=M_{e^{\mathrm{g}}}\) is the operator of multiplication by \(e^{\mathrm{g}}\),
\begin{equation*}
\mathrm{g}(z)= -\frac{iu}{2stz^2}
+\left(\frac{1}{s}\right)\frac{1}{z}-\frac12 \log\left(1-\frac{iu}{zt}\right)\,.
\end{equation*}

By Taylor expansion, we obtain
$$ e^{c\alpha_k}\Gamma_\pm(x) e^{-c\alpha_k}=e^{\pm cx^k}\Gamma_\pm(x)\, .$$
Hence, the conjugation by $\mathrm{S}$ produces the following result:
$$\oint\frac{dy}{y} \sqrt{\frac{w}{y}}
\exp(\mathrm{g}(1/y)-\mathrm{g}(1/w))\Gamma_+(y)\Gamma_-(w)\, .$$

Finally, we rewrite the integral in terms of semiforms.
Indeed, the  implicit equation for $w$ and $y$ implies:
$$dy\left(1-\frac{t}{iuy} \right)=dw\left(1-\frac{t}{iuw}\right)\, .$$
On the other hand, we have
$$e^{\mathrm{g}(1/z)}=\exp\left[-\frac{iuz^2}{2st}
+\frac{z}{s}\right]\left(1-\frac{iuz}{t}\right)^{-1/2}\, .$$
We obtain
$$e^{xs\mathrm{D}}=\oint\frac{\sqrt{dydw}}{y} \Gamma_+(y)\Gamma_-(w)\, ,
\quad \quad w e^{-iuw/t}=ye^{-iuy/t}e^{-sx}\, .$$
Combining the above equation with 
$$ \Gamma_+(y)\Gamma_-(w)=:\Gamma_+(y)\Gamma_-(w):/(1-w/y)$$
yield the formula in the statement of the Lemma. \qed

\subsection{The 1-leg case}
We prove here a weaker 1-leg version of Theorem \ref{thm:two_leg}.
 
\begin{theorem} \label{thm:one_leg}  After the  change of variables 
$q=-e^{iu}$, 
we have for any collection of \(k_i\ge 0\)
 $$ \Big\langle\prod_i \Hr^{\GW}_{k_i}({\mathsf{p}})\Big|\mu,\emptyset,\emptyset\Big\rangle^{\GW,\mathsf{T}_0,\bullet}_{U,D}= q^{-|\mu|}
 \Big\langle\prod_i \Hr^{\DT}_{k_i}({\mathsf{p}})\Big|\mu,\emptyset,\emptyset\Big\rangle^{\DT,\mathsf{T}_0}_{U,D},$$
 where \(\mathsf{T}_0\subset \mathsf{T}\) is the subtorus preserving the symplectic form on \(\CC^2\).
\end{theorem}

\begin{proof}
 We start by rewriting Proposition \ref{BosonFermion} after
the change of variables $$y \mapsto i\frac{uy}{t}\,,  \quad \quad 
w\mapsto i\frac{uw}{t}$$ in the form
\begin{equation}\label{eq:expD}
  e^{xs\mathrm{D}}=\oint_{|y|=1/\epsilon}\frac{\sqrt{dydw}}{y-w} :\exp\Bigg(\frac{\bar{\phi}(y)}{s}-\frac{\bar{\phi}(w)}{s}\Bigg):
\end{equation}
\begin{equation*}
\bar\phi(z)=\sum_{n>0}\frac{\bar{\alpha}_n}{n} \left(\frac{izt}{u}\right)^{-n}+\frac{1}{t}\sum_{n<0}\frac{\bar{\alpha}_n}{n} \left(\frac{izt}{u}\right)^{-n}\, .
\end{equation*}
where \(y,w\) are constrained by \(ye^y=we^we^{sx}\) and
\[\bar{\alpha}_k=s\alpha_k,\quad \bar{\alpha}_{-k}=st\alpha_{-k}-t\delta_{k-1}+\delta_{k-2}iu,\quad k>0.\]

We define a homomorphism{\footnote{The equivariant cohomology of
$U$ is generated over $\mathbb{Q}[s,t]$ by the class $\mathsf{p}$
of the fixed point.}} 
\(\mathsf{F}: \mathsf{Heis}\rightarrow \mathsf{Heis}_U\) by
\[\mathsf{F}(\bar{\alpha}_k)=\mathfrak{a}_k(\mathsf{p})\, ,\quad \quad k\in \ZZ\setminus \{0\}\, .\]
The linear map \(\mathsf{F}\) is a homomorphism of algebras because
\[[\mathsf{F}(\bar{\alpha}_k),\mathsf{F}(\bar{\alpha}_m)]=
[\mathfrak{a}_k(\mathsf{p}),\mathfrak{a}_m(\mathsf{p})]=k\delta_{k+m}\int 
\mathsf{p}\cdot \mathsf{p}=k\delta_{k+m}(-s^2t)\, .\]
Moreover, $\mathsf{F}$ sends the LHS of (\ref{eq:expD}) to the LHS of (\ref{eq:H^GW}) since 
$$t=c_1(T_U)\,,  \quad \quad s^2=-c_2(T_U)\, .$$

Let \(\mathcal{F}\) be the standard Fock space for \(\mathsf{Heis}\) with the vacuum vector \(v_\emptyset\),
\[\alpha_kv_\emptyset=0, \quad k<0.\] 
We denote by \(\mathcal{F}_{geom}\) the Fock space space defined by the action
of \(\mathsf{Heis}_U\)
 on the vacuum:
\[\mathfrak{a}_{k}({\mathsf{p}})u_\emptyset=\left[\int \big(-t\delta_{k+1}+\delta_{k+2}iu\big)\cdot {\mathsf{p}} \right] u_\emptyset\, ,\quad k<0\, .\]
%On the other hand, let \(\mathcal{F}\) be the standard Fock space for \(\mathsf{Heis}\) with the vacuum vector \(v_\emptyset\),
%\[\alpha_kv_\emptyset=0, \quad k<0.\] 
The homomorphism \(\mathsf{F}\) induces a canonical 
homomorphism of the Fock space,
\begin{equation}\label{34pp}
\mathsf{F}: \mathcal{F} \rightarrow \mathcal{F}_{geom}\, ,
\end{equation}
 by matching
vacuum vectors \(\mathsf{F}(v_\emptyset)=u_\emptyset\) since
$$\mathsf{F}(\bar{\alpha}_k) u_\emptyset=\mathfrak{a}_k u_\emptyset \, ,\quad k<0\, .$$

The Fock space \(\mathcal{F}_{geom}\) has a natural linear functional which
 evaluates Gromov-Witten invariants. Since the elements \[
\prod_{i>0}\mathfrak{a}_i^{k_i}(\mathsf{p}) u_\emptyset\] form a basis of \(\mathcal{F}_{geom}\), we have a natural linear isomorphism between
\(\mathsf{Heis}_U^+\) and \(\mathcal{F}_{geom}\). The isomorphism allow us to define
the following linear functional on \(\mathcal{F}_{geom}\):
\begin{equation}\label{linfun}
\Psi_\mu^{\GW}(\Phi)=\big\langle \Phi\big|\mu\big\rangle^{\mathsf{GW},\mathsf{T}_0,\bullet}_\beta\, ,
\end{equation}
where \(\beta=|\mu|\proj^1\).
Formula \eqref{eq:evGW} implies that under the 
identification \eqref{34pp}  of
the Fock spaces \(\mathcal{F}\) and \(\mathcal{F}_{geom}\), the  linear
functional \eqref{linfun} corresponds to
a pairing in $\mathcal{F}$ 
with the vector 
$$v_\mu=\overline{W}^{-1}e^{\alpha_1}|\mu\rangle\in \mathcal{F}\, .$$

On the stable pairs side, equation \eqref{eq:evPT} evaluates the the
right side of the correspondence of Theorem \ref{thm:one_leg}:
$$q^{-|\mu|}\Big\langle\prod_{j=1}^\ell\Hr^{\DT}_{k_j}([0])\Big|\mu\Big\rangle^{\DT,\mathsf{T}_0}_X=(-s)^\ell[x^{\vec{k}}]
\Big\langle\prod_{j=1}^\ell e^{sx_j\mathrm{D}}
\overline{W}^{-1}e^{\alpha_1}\Big|\mu\Big\rangle^{\mathcal{F}}\, .$$
Since $\mathsf{F}$ sends the LHS of (\ref{eq:expD}) to the LHS of (\ref{eq:H^GW}),
we obtain
$$(-s)^\ell[x^{\vec{k}}]\Big\langle\prod_{j=1}^\ell e^{sx_j\mathrm{D}}
\overline{W}^{-1}e^{\alpha_1}\Big|\mu\Big\rangle^{\mathcal{F}} =
\Big\langle\prod_{j=1}^\ell \Hr^{\GW}_{k_i}({\mathsf{p}})\Big|\mu,\emptyset,\emptyset\Big\rangle^{\GW,\mathsf{T}_0,\bullet}_{X}$$
via the evaluation of $\Psi_\mu^{\GW}$ as a pairing in $\mathcal{F}$ 
with $v_\mu$.
\end{proof}

The following two remarks about Theorem \ref{thm:one_leg} will be important
for our further study of the descendent correspondence in
Section \ref{sec:uniq-cor}:

\begin{enumerate}
\item[(i)] The $\bullet$-series of Theorem \ref{thm:one_leg}
  agrees exactly with the Gromov-Witten series defined in Section \ref{ccppvv} with
  {\em no collapsed connected components of genus greater than or equal to 2},
  $$ \Big\langle\prod_i \Hr^{\GW}_{k_i}({\mathsf{p}})\Big|\mu,\emptyset,\emptyset\Big\rangle^{\GW,\mathsf{T}_0,\bullet}_{U,D}=
  \Big\langle\prod_i \Hr^{\GW}_{k_i}({\mathsf{p}})\Big|\mu,\emptyset,\emptyset\Big\rangle^{\GW,\mathsf{T}_0}_{U,D}\, .$$
  The above equality follows from the formula for the virtual class of
  the moduli space of maps of a
  collapsed connected components{\footnote{See Section 2.3 of \cite{GetPan}.}} of genus $g\geq 2$,
  $$ (-1)^g\, (c_3(T_{U/D})-c_2(T_{U/D})c_1(T_{U/D})) \, \lambda_{g-1}^3\, $$
which vanishes for the torus $\mathsf{T}_0$.

  \item[(ii)]
    Since the Gromov-Witten bracket is compatible with the hat operation
\eqref{fredfredfred},
we can equivalently write the conclusion of Theorem \ref{thm:one_leg}
after applying (i) 
as:
\begin{equation}
\label{yayay} 
 \Big\langle\prod_i \widehat{\Hr}^{\GW}_{k_i}({\mathsf{p}})\Big|\mu,\emptyset,\emptyset\Big\rangle^{\GW,\mathsf{T}_0}_{U,D}= q^{-|\mu|}
 \Big\langle\prod_i \Hr^{\DT}_{k_i}({\mathsf{p}})\Big|\mu,\emptyset,\emptyset\Big\rangle^{\DT,\mathsf{T}_0}_{U,D}\, .
\end{equation}
\end{enumerate}

\subsection{Lambert function}
\label{sec:labmert-function}

%Before we discuss the properties of the GW/PT correspondence of Section
%\ref{einbein},
We explain how  to convert the contour integral in definition
 \eqref{eq:H^GW} to an explicit
formula. The first step is to 
solve the constraint equation (\ref{eq:master_curve}),
\begin{equation*}
ye^y=we^w e^{-x/\theta}\, .
\end{equation*}
We interpret both sides as formal power series in \(x\), and
 then we can find the solution
by induction on degree of \(x\). 
In particular, the first few terms of the expansion are:
 \begin{equation}\label{ww33ww}
w(y)=y-\frac{xy}{\theta(y+1)}+\frac{x^2y}{2\theta^2(y+1)^3}+\frac{x^3y(2y-1)}{6\theta^3(y+1)^5}+O(x^4)\, .
 \end{equation}

% Function \(xe^x\) is invertible in the small disk around \(x=0\). The inverse function is call {\it Lambert function} \(W(x)\) \cite{Lamb} and  of the power series expansion
% of \(W\) is the generating function for the numbers of trees:
%\[W(x)=\sum_{n=1}^\infty \frac{(-n)^{n-1}}{n!}x^n.\]
We  can therefore
write explicit power series for the integrand in  formula  \eqref{eq:H^GW}
 and find an effective formula  for \(\Hr^{\GW}(x)\):
 \begin{equation}
   \label{eq:LambertH^GW}
\frac{x}{\theta}   \mathrm{Res}_{y=\infty}\left(dy
     \left(\frac{dw(y)}{dy}\right)^{1/2}\frac{:e^{\theta(\phi(y)-\phi(w(y)))}:}{y-w(y)}\right),
 \end{equation}
 where  %\(w(y)=W(ye^y e^{-x/\theta})\) and
$w(y)$ is given by \eqref{ww33ww} and \(\phi(z)\) is by \eqref{eq:phiz}.

\section{Uniqueness of the correspondence}
\label{sec:uniq-cor}

\subsection{Properties of the correspondence matrix}
\label{sec:uniq-corr-1}

We define an augmented partition size $|\cdot |^+$ by the formula
$$|\lambda|^+=\sum_{i}(1+\lambda_i)\, .$$
Let  \(\mathcal{P}\) be the set of all partitions.
Let \(\mathcal{P}_d\) the set of partitions of 
augmented size \(d\), and let
\(\mathcal{P}_{\le d}\) be the set of partitions of augmented
size less than or equal to \(d\),
$$\mathcal{P}_d \subset \mathcal{P}_{\le d} \subset \mathcal{P}\, .$$
As in Section \ref{sec:non-equiv-limit}, we set 
$$\Hr^{\mathsf{GW}}_\mu=\prod_{i=1}^\ell \Hr^{\mathsf{GW}}_{\mu_i}\in \mathsf{Heis}^c\, , \quad \quad 
\widehat{\Hr}^{\GW}_\mu\in \mathsf{Heis}^c_+\otimes\CC[c_1,c_2^{1/2}]\, . $$

\begin{Lemma} \label{xxdd22}
For every \(\mu\in \mathcal{P}_d\),  we have
  \begin{enumerate}
  \item[(i)] \(\widehat{\Hr}^{\GW}_\mu\in \mathsf{Heis}^c_+\otimes \CC[c_1,c_2]\) ,

\vspace{5pt}
  \item[(ii)] \(\widehat{\Hr}^{\GW}_\mu=\frac{\mathfrak{a}_\mu}{(\mu-1)!}+\sum_{\lambda\in \mathcal{P}_{<d}} b(\mu,\lambda)\mathfrak{a}_\lambda\),
  \end{enumerate}
  with \(\mathfrak{a}_{\lambda}=\prod_{i=1}^\ell\mathfrak{a}_{\lambda_i}\).
\end{Lemma}
\begin{proof}
  The operator \(s\mathrm{D}\) defined by \eqref{eq:opD}
 is a linear combination of  monomials in \(H\) and \(\alpha_k\) with coefficients in \(\CC[s,t]\).
  The same holds for every power of \(s\mathrm{D}\). Since the operator \( H\) is a quadratic expression\footnote{\(H=\sum_{k>0} \alpha_{-k}\alpha_{k}\).} of \(\alpha_k\) with coefficients
  in \(\CC[s,t]\),
 we conclude 
 $$\widehat{\Hr}^{\GW}(x)\in \mathsf{Heis}^c\otimes \CC[c_1,c_2^{1/2}]\, .$$

  The integral defining \(\widehat{\Hr}^{\GW}(x)\) is invariant
  with respect to the sign change $$\theta\mapsto -\theta\, .$$
 Indeed, under the sign change the constraint equation turns into
  $$ze^z=we^we^{-x/\theta}$$
  which is equivalent to the original constraint equation after 
 switching \(y\) and \(w\). On the other hand,
 the integral is unchanged after the
  switch. Thus, we have proven claim (i).

  Definition \eqref{eq:opD} is homogeneous for the 
 homological grading of the generators:
  \[\deg \alpha_k=k+1,\quad \deg \alpha_{-k}=-k+2,\quad \deg s=\deg t=1\, .\]
 The powers of \(\mathrm{D}\) are therefore also homogeneous, and claim (ii)
 follows.  
\end{proof}

\subsection{Uniqueness}
\label{sec:uniqueness}

%Let us use notations \(\mathcal{P}_{\le d}\) and \(\mathcal{P}\) for the vector spaces with basis labeled by the sets of corresponding partitions. 

The 1-leg $\GW/\PT$ descendent correspondence of Theorem \ref{thm:one_leg}
in the hat form of \eqref{yayay} is
\begin{equation}\label{crru}
\Hr_\mu^\PT \mapsto \widehat{\Hr}_\mu^\GW\, .
\end{equation}
The correspondence rule \eqref{crru} defines a \(\CC[c_1,c_2]\)-linear
operator 
$$\mathcal{T}:\mathcal{P}^{\PT}\rightarrow \mathcal{P}^{\GW}\, ,$$
 where \(\mathcal{P}^{\PT}\) has \(\CC[c_1,c_2]\)-basis \(\Hr_\mu^{\PT}\) and
\(\mathcal{P}^{\GW}\) has \(\CC[c_1,c_2]\)-basis \(\mathfrak{a}_\mu\).
By Lemma \ref{xxdd22} part (ii), $\mathcal{T}$ restricts to 
$$\mathcal{T}_d:\mathcal{P}^{\PT}_{\le d}\rightarrow\mathcal{P}^{\GW}_{\le d}\, ,$$
where the shifted size of the partitions are bounded in the bases
on both sides.

There are two operators that encode 1-leg relative Gromov-Witten
and stable pairs theories. 
The first operator 
$$\mathcal{M}^{\GW}_d:\mathcal{P}^{\GW}_{\le d}\rightarrow \mathcal{P}$$
has the Gromov-Witten invariants 
$$\big\langle \mathfrak{a}_\mu(\mathsf{p})\big|\lambda\big\rangle^{\GW,\mathsf{T}_0}_{U,D}
\,, \quad
\mu\in \mathcal{P}_{\le d}\,, \quad \lambda\in\mathcal{P}$$
as matrix entries.
The second operator 
$$\mathcal{M}^{\PT}_d:\mathcal{P}^{\PT}_{\le d}\rightarrow \mathcal{P}$$
has the stable pairs invariants 
$$\big\langle  \Hr^{\PT}_\mu(\mathsf{p})   \big|\lambda\big\rangle^{\PT,\mathsf{T}_0}_{U,D}
\,, \quad
\mu\in \mathcal{P}_{\le d}\,, \quad \lambda\in\mathcal{P}$$
as matrix entries.

% The second operator \(M^{\PT}_d:\mathcal{P}^{\PT}\rightarrow\mathcal{P}\) has invariants
% \(\langle \Hr^{\PT}_\mu(\mathsf{p})|\lambda\rangle^{\mathsf{T}_0,\PT}_{U,D}\),
% \(\mu\in \mathcal{P}_{\le d}\), \(\lambda\in\mathcal{P}\) as matrix entries.

 \begin{Lemma}\label{lem:uniq}
   The operator \(\mathcal{T}_d\) is an isomorphism and is the
unique solution of the
correspondence equation
   \[\mathcal{M}^{\GW}_d \mathcal{T}_d=\mathcal{M}_d^{\DT}\, .\]
       \end{Lemma}
    \begin{proof}
      That $\mathcal{T}_d$ is an
isomorphism follows from Lemma \ref{xxdd22} part (ii).
The correspondence equation is
 exactly the statement of Theorem \ref{thm:one_leg} in form \eqref{yayay}. 

To derive uniqueness, we will show that the operator  
\(\mathcal{M}_d^{\DT}\)  is injective.
     By the construction of the projective representation \(\Lambda^{\infty/2}V\) (see for example \cite[section 2.2.2]{OPH/GW}), we have:
      \[\big\langle \Hr^{\DT}_{\mu}(p)\big|\lambda\big\rangle_{U,D}^{\PT,\mathsf{T}_0}=\mathbf{p}_\mu(\lambda)\, ,\]
      where \(\mathbf{p}_\mu=\prod_i\mathbf{p}_{\mu_i}\) is the product of the shifted Newton polynomials from the ring of the shifted symmetric functions
      \(\Lambda^*=\mathbb{Q}[\mathbf{p}_1,\mathbf{p}_2,\dots]\) and
      \[\mathbf{p}_k(\lambda)=\sum_{i=1}^\infty \left[(\lambda_i-i
+\frac{1}{2})^k-(-i+\frac{1}{2})^k\right]+(1-2^{-k})\zeta(-k),\]
      is the evaluation of the shifted function at \(\lambda\).
      
      Since the products \(\mathbf{p}_\mu\) span a basis of the ring of the shifted symmetric functions and the evaluation map 
$$f\mapsto \{f(\lambda)\}_{\lambda\in\mathcal{P}}$$ is
    the  Fourier transform in representation theory of \(S_\infty\) \cite{OkOl},
 ${\mathcal{M}}^\PT_d$ is  injective.
    \end{proof}

\subsection{Comparing correspondences}
\label{sec:comp-corr}
The $\GW/\PT$ correspondence for the standard descendents \(\tau_k\)
is studied in \cite{PPDC}. Since
the descendents \(\mathfrak{a}_k\) and \(\Hr^{\DT}_k\) are the linear combinations of the standard descendents $\{ \tau_m \}_{m\le k}$, the  results of \cite{PPDC}
hold in our setting here.

\begin{theorem}\cite{PPDC} \label{pppddd} There exists an invertible transformation 
 \(\underline{\mathcal{T}}:\mathcal{P}^{\PT}\rightarrow\mathcal{P}^{\GW}\) linear over \(\CC[s_1,s_2,s_3]\) for which the correspondence equation
  \[\big\langle \underline{\mathcal{T}}(\mu)(\mathsf{p})\big|\lambda_1,\lambda_2,\lambda_3\big\rangle_{U,D}^{\GW,\mathsf{T}}
=
\big\langle \Hr^{\PT}_\mu(\mathsf{p})\big|\lambda_1,\lambda_2,\lambda_3\big\rangle_{U,D}^{\PT,\mathsf{T}}
\]
holds  for all \(\mu,\lambda_1,\lambda_2,\lambda_3\in\mathcal{P}\). Moreover,
\vspace{4pt} 
 \begin{enumerate}
  \item[(i)] $\underline{\mathcal{T}}$
    sends \(\mathcal{P}^{\PT}_{\le d}\) to \(\mathcal{P}^{\GW}_{\le d}\),
\vspace{4pt}
  \item[(ii)] the coefficients of $\underline{\mathcal{T}}$ are  polynomials in
the symmetric functions $$c_1=e_1(s_1,s_2,s_3)\, , \ \ c_2=e_2(s_1,s_2,s_3)\, ,\ \ c_3=e_3(s_1,s_2,s_3)\, .$$
    \end{enumerate}
\end{theorem}

\begin{proof}
  Theorems 1-3 of \cite{PPDC} prove the parallel statement of Theorem
\ref{pppddd}
  for a slightly different generating
  function of the invariants on the Gromov-Witten side:
\begin{equation}\label{gg234g}
  \big\langle\prod_{i=1}^m\tau_{k_i}(\gamma_i)|\, \lambda\big\rangle_{g,U,D}^{\GW,\mathsf{T},*}\, .
  \end{equation}
%where \(D\) is a  nonsingular divisor in \(X\), \(\lambda\in
%  H^*_{\mathsf{T}}(\mathsf{Hilb}_{D\cdot\beta}(D))\) and \(\gamma_i\in H^*_\mathsf{T}(X)\).
  The invariant \eqref{gg234g}
  is defined as an integral over the virtual cycle of the moduli space
  \(\overline{M}_{g,m}^*(U/D,\beta)\) of stale relative maps
  from a possibly disconnected curve of
  genus \(g\) to the pair $(U,D)$ with {\em
    no connected collapsed components}.

  By Theorems 1-3 of
\cite{PPDC}, there exists an invertible transformation 
 \(\underline{\mathcal{T}}^*:\mathcal{P}^{\PT}\rightarrow\mathcal{P}^{\GW}\) linear over \(\CC[s_1,s_2,s_3]\) for which the correspondence equation
  \[\big\langle \underline{\mathcal{T}}^*(\mu)(\mathsf{p})\big|\lambda_1,\lambda_2,\lambda_3\big\rangle_{U,D}^{\GW,\mathsf{T},*}
=
\big\langle \Hr^{\PT}_\mu(\mathsf{p})\big|\lambda_1,\lambda_2,\lambda_3\big\rangle_{U,D}^{\PT,\mathsf{T}}
\]
holds  for all \(\mu,\lambda_1,\lambda_2,\lambda_3\in\mathcal{P}\).
Moreover, \(\underline{\mathcal{T}}^*\) satisfies
the above condition (i) and (ii).

The Gromov-Witten invariants
\eqref{gg234g} of \cite{PPDC} are related to the Gromov-Witten invariants \eqref{j233}  
defined in Section \ref{ccppvv} by the
following simple equation for every target geometry $(X,D)$.
We sum over
all set partitions $$\{1,\dots,m\}=P'\sqcup P''\sqcup P'''$$
and  refinements
  $P'=\sqcup_{a=1}^A P'_a$, $P''=\sqcup_{b=1}^B P''_b$ :
  \[\big\langle\prod_{i=1}^m\tau_{k_i}(\gamma_i)|\, \mu\big\rangle_{X,D}^{\GW,\mathsf{T}}=\sum_{P}u^{-2A}\, C^0_{P'}C^1_{P''}\,
      \big\langle\prod_{i\in P'''}\tau_{k_i}(\gamma_i)
      |\, \mu\big\rangle_{X,D}^{\GW,\mathsf{T},*}, \]

\begin{eqnarray*}
    C^0_{P'}&=& \prod_{a=1}^A\, \int_X\gamma_{P'_a}\cdot \int_{\overline{M}_{0,|P'_a|}}
    \, \prod_{i\in P'_a}\tau_{k_i}\,,  \\
  C^1_{P''}&=&\prod_{b=1}^B\bigg( \int_X c_1(T_{X/D})\cdot \gamma_{P''_b} \cdot
               \int_{\overline{M}_{1,|P''_b|}} \, \prod_{i\in P''_b}\tau_{k_i}
               -\int_X c_2(T_{X/D})\cdot \gamma_{P''_b}\cdot
               \int_{\overline{M}_{1,|P''_b|}}c_1(\mathbb{E})
  \prod_{i\in P''_b}\tau_{k_i}\bigg)\, .
\end{eqnarray*}
In the formula, the moduli spaces of stable curves of genus 0 and 1
which appear correspond exactly to collapsed components.
In genus 1, $\mathbb{E}$ is Hodge bundle.
As before, we use the convention
$$\gamma_S=\prod_{i\in S}\gamma_i\, .$$

Hence, there exists an invertible transformation 
$$\underline{\mathcal{R}}:\mathcal{P}^{\GW}\rightarrow\mathcal{P}^{\GW}$$
linear over \(\CC[s_1,s_2,s_3]\) for which the correspondence equation
  \[\big\langle \underline{\mathcal{R}}(\mu)(\mathsf{p})\big|\, \lambda_1,\lambda_2,\lambda_3\big\rangle_{U,D}^{\GW,\mathsf{T},*}
=
\big\langle \tau_\mu(\mathsf{p})\big|\, \lambda_1,\lambda_2,\lambda_3\big\rangle_{U,D}^{\GW,\mathsf{T}}
\]
holds  for all \(\mu,\lambda_1,\lambda_2,\lambda_3\in\mathcal{P}\).
Moreover, the  matrix \(\underline{\mathcal{R}}\) satisfies
conditions (i) and (ii) of  Theorem \ref{pppddd}. 
The composition
$$\underline{\mathcal{T}}=\underline{\mathcal{T}}^*\circ \underline{\mathcal{R}}$$
is therefore the sought after linear transformation.
\end{proof}

\begin{Corollary} \label{coco9}The coefficients of $\mathcal{T}$ are polynomial in \(c_1,c_2\) and
  \[ {\mathcal{T}}=\underline{\mathcal{T}}|_{c_3=c_1c_2}\, .\]
  
\end{Corollary}
\begin{proof}
  The uniqueness of Lemma \ref{lem:uniq} implies that 
$\mathcal{T}=\underline{\mathcal{T}}|_{s_1=-s_2}$. 
Hence, the coefficients of $\mathcal{T}$ are polynomials of \(s=s_1\) and \(t=s_3\).
 Since
  \[ c_1|_{s_1=-s_2}=t, \quad c_2|_{s_1=-s_2}=-s^2, \]
  and since $\underline{\mathcal{T}}$ is symmetric with respect to all
permutations of \(s_i\), the coefficients of $\mathcal{T}$ must be
 polynomial in \(c_1,c_2\).
\end{proof}

\subsection{Poles}

The following pole restriction result will play a crucial role
in the proof of Theorem \ref{thm:two_leg} in Section 
\ref{sec:two-leg-case}.

\begin{Lemma}\label{lem:poles}
  The descendent invariants 
  \[
\big\langle\tau_\mu\big|\lambda_1,\lambda_2,\lambda_3\big\rangle^{\GW,\mathsf{T}}_{U,D} \quad \text{and} \quad
\big\langle\mathrm{ch}_\mu\big|\lambda_1,\lambda_2,\lambda_3\big\rangle^{\PT,\mathsf{T}}_{U,D}
\] have no poles along the hyperplane \(s_i+s_j=0\) if either
$\lambda_i=\emptyset$ or $\lambda_j=\emptyset$.
%\(|\lambda_i|\cdot|\lambda_j|=0\).
\end{Lemma}
\begin{proof}
  Since the matrix \(\underline{\mathcal{R}}\) from the proof of Theorem~\ref{pppddd} has the required pole property, it is enough
  to prove the statement for the invariants  
    \[
\big\langle\tau_\mu\big|\lambda_1,\lambda_2,\lambda_3\big\rangle^{\GW,\mathsf{T},*}_{U,D} \quad \text{and} \quad
\big\langle\mathrm{ch}_\mu\big|\lambda_1,\lambda_2,\lambda_3\big\rangle^{\PT,\mathsf{T}}_{U,D}.
\]

  The invariants here are the capped vertices \cite{MOOP,PPDC}.
  The stated regularity property for Gromov-Witten
 invariants follows from the localization formula \cite{GP} 
for the capped vertex \cite[section 2]{MOOP}.
  As explained in \cite{MOOP},
  \begin{multline}\label{eq:loccap}
    \big\langle\tau_\mu\big|\lambda_1,\lambda_2,\lambda_3\big\rangle^{\GW,\mathsf{T},*}_{U,D}=\sum_{\lambda'_1,\lambda'_2,\lambda'_3} \mathsf{V}_{\GW}(\tau_\mu|\lambda'_1,\lambda'_2,\lambda'_3,u)
    \cdot H(\lambda'_k,s_{k+1},s_{k+2},s_k)\\
\cdot \prod_{k=1}^3\Psi_{\GW}(\lambda_k,\lambda'_k,s_{k+1},s_{k+2},-s_k,u)\, ,
\end{multline}
  where the partitions
in the sum are constrained by \(|\lambda'_i|=|\lambda_i|\), 
the half-edge term \(H\) is the edge-term for the local curve 
theory \cite{BP}, 
the term 
  \[\Psi_{\GW}(\lambda,\mu,s_1,s_2,s_3,u)=\sum_g\bigg\langle \lambda\bigg| \frac{1}{s_3-\psi_\infty}\bigg|\mu\bigg\rangle^{\sim \empty \prime}_{g,d}u^{2g-2}\, ,\]
is the rubber integral,
  and \(\mathsf{V}_{\GW}(\tau_\mu|\lambda_1,\lambda_2,\lambda_3,u)\) is the standard localization vertex \cite{GP} in Gromov-Witten theory.

  The rubber integral is regular at \(s_i+s_j=0\), the half-edge term is the ratio of the explicit products of the linear expressions of
  \(s_i\) which can be easily checked to be regular at \(s_i+s_j=0\).
The only potential source of poles at \(s_i+s_j=0\) is
 the standard localization vertex \(\mathsf{V}_{\GW}(\tau_\mu|\lambda_1,\lambda_2,\lambda_3)\).
 
 The standard Gromov-Witten localization vertex \(\mathsf{V}_\GW(\tau_\mu|\lambda_1,\lambda_2,\lambda_3)\) is straightforward to analyze directly
from the formula of \cite{GP}.
%
% the sum over the marked graphs
%  \(\Gamma\) \cite[chapter 27]{Mir}. The graph \(\Gamma\) contains the information
%  about the topology of the source curve and about the images of the contracted
%  components of the curve and \(M_\Gamma\subset \overline{M}_{g,\ell}(U/D)^{\ma%thsf{T}}\), \(\ell=\ell(\mu)\) is the
%  corresponding connected component of the torus-fixed locus.
%  The standard GW localization vertex  the sum over all possible geometrically %admissible marked
%  graphs of the integrals:
%  \[\int_{M_\Gamma}\frac{\prod_i\psi_i^{\mu_i}}{e(N^{vir}_{M_\Gamma})},\]
%  where the \(N^{vir}_{M_\Gamma}\) is the virtual normal bundle to the \(\maths%f{T}\)-fixed locus.
%The 
%From the formula for the Euler class \(e(N^{vir}_{M_\Gamma})\), \cite[chapter 27]{Mir} one can see that 
%
In fact, the only
  source of poles at \(s_i+s_j=0\) is the tangent weight 
of the tangent space
  the space of smoothing of a nodal rational curve
(which occurs in the Euler class of the virtual normal bundle
to the $\mathsf{T}$-fixed locus.
If we are smoothing
  a node connecting the rational components with \(\mathsf{T}\)-weights 
\(\frac{s_i}{d_i}\)
  and \(\frac{s_j}{d_j}\) at the node then the tangent space to the smoothing family is
  $$\frac{s_i}{d_i}+\frac{s_j}{d_j}\, .$$
  Since \(d_i\) and \(d_j\) are the degrees of the images of the corresponding
  rational components, 
we have \(d_i\le |\lambda_i|\) and \(d_j\le |\lambda_j|\).
  Thus the pole statement follows in the Gromov-Witten case
since at least one of $\lambda_i$ and $\lambda_j$ are assume to
be empty.

  The $\PT$ case is shown by a computation similar to
 \cite[Section 3.3]{MO} where the parallel
  $\mathsf{DT}$ statement is proven. In \cite{MOOP}, 
the formula for
  $$\big\langle\mathrm{ch}_\mu\big|\lambda_1,\lambda_2,\lambda_3\big\rangle^{\PT,\mathsf{T}}_{U,D}$$
 analogous to (\ref{eq:loccap}) is written.
  It immediately follows that the only possible source of poles at \(s_i+s_j=0\) is the standard localized vertex
  \(\mathsf{V}_{\PT}(\tau_\mu|\lambda_1,\lambda_2,\lambda_3)\) for
  $\PT$ theory \cite{pt2}.
Thus, we must analyze the poles of 
$$\mathsf{V}_\PT(\mathrm{ch}_\mu|\lambda_1,\emptyset,\lambda_3)$$
 along \(s_1+s_2=0\). We will use
  the rim-hook  technique of \cite{MNOP1}.

Let us recall the basic structure of the standard
$\PT$ localization vertex
  from \cite{pt2}. 
  To  a partition \(\lambda_i\), we attach a monomial ideal \(\lambda_i[x_{i-1},x_{i+1}]\subset\CC[x_{i-1},x_{i+1}]\) and \(\CC[x_1,x_2,x_3]\)-modules
  \[M_i=\CC[x_{i},x_{i}^{-1}]\otimes \frac{\CC[x_{i-1},x_{i+1}]}{\lambda_i[x_{i-1},x_{i+1}]},\quad \quad M=\bigoplus_{i=1}^3 M_i\, .\]
The \(\mathsf{T}\)-fixed points of
the moduli space of stable pairs \(P_\bullet(U/D)_{\lambda_1,\lambda_2,\lambda_3}\)
  correspond to finitely generated \(\mathsf{T}\)-invariant \(\CC[x_1,x_2,x_3]\)-submodules:
  \[Q\subset M/\langle (1,1,1)\rangle\, .\]
In the case at hand, \(\lambda_2=\emptyset\), so 
a\cite{pt2} 
the \(\mathsf{T}\)-invariant submodules \(Q\) as above
  form 0-dimensional families \cite{pt2}.
We can choose a monomial basis for each such \(Q\). 
The combinatorics of the \(\mathsf{T}\)-weights
  of a monomial basis of \(Q\) is discussed below.

  The \(\mathsf{T}\)-weights of the homogeneous monomials inside \(M_i\) form an infinite cylinder $$\mathrm{Cyl}_i\subset\ZZ^3\, .$$
  Since \(\lambda_2=\emptyset\), the cylinder \(\mathrm{Cyl}_2\) is empty.
Hence, the weights of \(Q\) form some subset of
  the union \(\mathrm{Cyl}_1\cup\mathrm{Cyl}_3\). The union has three types of weights:
  \[\mathrm{Cyl}_1\cup \mathrm{Cyl}_3=\textup{\uppercase\expandafter{\romannumeral 1}}^+\cup \textup{\uppercase\expandafter{\romannumeral 2}}\cup \textup{\uppercase\expandafter{\romannumeral 1}}^-,\]
  where \(\textup{\uppercase\expandafter{\romannumeral 2}}=\mathrm{Cyl}_1\cap\mathrm{Cyl}_3\),
  \(\textup{\uppercase\expandafter{\romannumeral 1}}^+\) consists of the weights that have only non-negative coordinates and lie in exactly
  one cylinder,  and
\(\textup{\uppercase\expandafter{\romannumeral 1}}^-\) are the rest of the weights.

  The submodule \(Q\) is uniquely characterized by the
associated set of weights \(\mathsf{wt}(Q)\). Conversely, 
a subset \(S\subset\ZZ^3\) is
  a set of weights of \(Q\) corresponding to a \(\mathsf{T}\)-invariant element of \(P_n(X/D)_{\lambda_1,\emptyset,\lambda_3}\) if and only
  if the following three conditions holds:
  \begin{enumerate}
  \item[(i)] \(S\subset \textup{\uppercase\expandafter{\romannumeral 1}}^-\cup \textup{\uppercase\expandafter{\romannumeral 2}}\)
  \item[(ii)] \(w\in S\) if any of the weights
    \[(w_1-1,w_2,w_3),\quad (w_1,w_2-1,w_3),\quad (w_1,w_2,w_3-1)\]
    are in \(S\).
  \item[(iii)] \(|S|=n\).
    
  \end{enumerate}

  Let us call the set of weights as above {\it geometric. }
  For given a geometric  set of weights \(Q\) we introduce the generating functions:
  \[\mathsf{F}_0(Q)=\sum_{(ijk)\in Q} s_1^is_2^js_3^k+\sum_{(ijk)\in \textup{\uppercase\expandafter{\romannumeral 1}}^+}s_1^is_2^js_3^k,\]
  \[\mathsf{F}_{12}(Q)=\sum_{(ij)\in \lambda_3}s_1^is_2^j,\quad \mathsf{F}_{23}(Q)=\sum_{(ij)\in \lambda_1} s_2^is_3^j.\]
In \cite{pt2}, the generating function of the redistributed virtual weights of the normal bundle to
  the corresponding \(\mathsf{T}\)-fixed point of \(P_n(U/D)_{\lambda_1,\emptyset,\lambda_3}\) is defined by:
  \[\mathsf{V}_Q=\mathsf{F}_0-\frac{\overline{\mathsf{F}}_0}{s_1s_2s_3}+\mathsf{F}_0
    \overline{\mathsf{F}}_0\frac{(1-s_1)(1-s_2)(1-s_3)}{s_1s_2s_3}
    +\frac{\mathsf{G}_{12}}{1-s_3}+\frac{\mathsf{G}_{23}}{1-s_1},\]
  where \(\overline{f}(s_1,s_2,s_3)=f(s_1^{-1},s_2^{-1},s_3^{-1})\), \(\mathsf{F}_0=\mathsf{F}_0(Q)\), and
  \[\mathsf{G}_{ij}=-\mathsf{F}_{ij}-\frac{\overline{\mathsf{F}}_{ij}}{s_is_j}+\mathsf{F}_{ij}\overline{\mathsf{F}}_{ij}\frac{(1-s_i)(1-s_j)}{s_is_j},\]
  with \(\mathsf{F}_{ij}=\mathsf{F}_{ij}(Q)\).
  
  The standard localized vertex \(\mathsf{V}_{\PT}(\mathrm{ch}_\mu|\lambda_1,\emptyset,\lambda_3)\) is the sum over all geometric
  sets of weights \(Q\) of the expressions:
  \[ q^{|Q|} \prod_i \mathrm{ch}_{\mu_i}(\mathsf{F}_0(Q))\cdot e(-\mathsf{V}_Q).\]

  To prove the Lemma,  we must analyze the 
poles of \(e(-\mathsf{V}_Q)\), and we follow method of \cite{MNOP2} in 
our argument.
  The order of the pole at \(s_1+s_2=0\) is equal to  the constant term of \(\mathsf{V}_Q(x,x^{-1},t_3)\). 
  Substituting{\footnote{For shorter
formulas, we now drop \(Q\) from the notation.}}
  \[\mathsf{F}_0=\underline{\mathsf{F}}_0+\frac{\mathsf{F}_{23}}{1-s_1}\]
  into the the formula for \(\mathsf{V}\), we obtain
  \begin{multline}\label{eq:normV}
    \underline{\mathsf{F}}_0-\frac{\underline{\overline{\mathsf{F}}}_0}{s_1s_2s_3}+\underline{\mathsf{F}}_0
    \underline{\overline{\mathsf{F}}}_0\frac{(1-s_1)(1-s_2)(1-s_3)}{s_1s_2s_3}
    +\frac{\mathsf{G}_{12}}{1-s_3}+\\
    \underline{\overline{\mathsf{F}}}_0\mathsf{F}_{23}\frac{(1-s_2)(1-s_3)}{s_1s_2s_3}-
    \underline{\mathsf{F}}_0\overline{\mathsf{F}}_{23}\frac{(1-s_2)(1-s_3)}{s_2s_3}\, .
    \end{multline}

    Since \(\underline{\mathsf{F}}_0(x,x^{-1},s_3)\) has only strictly positive powers of \(x\) in its expansion
    and \(\mathsf{F}_{23}(x,s_3)\) has only positive powers of \(x\) in its expansion, we conclude that the functions
    \[\underline{\mathsf{F}}_0(x^{-1},x,s^{-1}_3)\mathsf{F}_{23}(x^{-1},s_3)\frac{(1-x^{-1})(1-s_3)}{s_3},\quad
      \underline{\mathsf{F}}_0(x,x^{-1},s_3)\mathsf{F}_{23}(x,s_3^{-1})\frac{(1-x^{-1})(1-s_3)}{x^{-1}s_3}\]
    have only strictly negative and strictly positive, respectively, powers of \(x\) in their expansions. The expression
    \(\mathsf{G}_{12}\) is the generating function for the tangent weights
    \(\mathsf{Hilb}_{|\lambda_3|}(\CC^2)\) at the corresponding monomial ideal, hence
    we can use well-known formula for the tangent weights to see that
    \(\mathsf{G}_{12}(x,x^{-1})\) has no constant term.

    To finish proof we must bound the constant term of 
    \begin{equation}\label{eq:CT}
    \underline{\mathsf{F}}_0(x,x^{-1},s_3)-\frac{\underline{\mathsf{F}}_0(x^{-1},x,s^{-1}_3)}{s_3}+\underline{\mathsf{F}}_0(x,x^{-1},s_3)
    \underline{\mathsf{F}}_0(x^{-1},x,s_3^{-1})\frac{(1-x)(1-x^{-1})(1-s_3)}{s_3}
  \end{equation}
  The function \(\underline{\mathsf{F}}_0(x,x^{-1},s_3)\) can be expanded in Laurent power series
  of \(s_3\),  and the coefficients of the expansion are Laurent polynomials in \(x\):
  \[\underline{\mathsf{F}}_0(x,x^{-1},s_3)=\sum_{i,j}a_{ij}x^is_3^j\, .\]
 The  formal computation of the proof of \cite[Lemma 5]{MNOP2}
  determines the constant term of \eqref{eq:CT} to be
  \[-\frac{1}{2}\sum_{i,j}\big((a_{i,j}-a_{i+1,j})-(a_{i,j+1}-a_{i+1,j+1})\big)^2\, ,\]
which is non-positive. \end{proof}

% Let us also remark that the argument from the last proof actually shows that the
% order of the pole along the hyperplane \(s_i+s_j=0\) of \(\big\langle\tau_\mu\big|\lambda_1,\lambda_2,\lambda_3\big\rangle^{\GW,\mathsf{T}}_{U,D}\)
% does not exceed \(min(|\lambda_i|,|\lambda_j|)\). The same estimate for the pole order holds for
% the DT integral, as  shown in \cite[section 3.3]{MO}. However, proving  this estimate for PT theory appears
% to be more difficult because the relevant torus fixed locus might  have positive dimension.
\subsection{Proof of Theorem \ref{thm:two_leg}}
\label{sec:two-leg-case}

By Theorem \ref{pppddd}, we have the correspondence
\begin{equation}\label{lldd4}
\big \langle \underline{\mathcal{T}}(\mu)(\mathsf{p})
\big|\lambda_1,\lambda_2,
\emptyset \big\rangle_{U,D}^{\GW,\mathsf{T}}
=
\big \langle \Hr^{\PT}_\mu(\mathsf{p})\big |\lambda_1,\lambda_2,\emptyset \big \rangle_{U,D}^{\PT,\mathsf{T}}\, .
\end{equation}
Theorem \ref{thm:two_leg} will be derived from 
equation \eqref{lldd4}.

The element \(\widehat{\Hr}^{\GW}_\mu(\mathsf{p})\) is a linear combination of monomials of \(\mathfrak{a}_i(\mathsf{p})\) with coefficients in
\(\CC[c_1,c_2]\). The descendents \(\Hr^{\PT}_\mu(\mathsf{p})\)
are the linear combinations of monomials of \(\text{ch}_i(\mathsf{p})\)
with coefficients in \(\CC[c_2]\).
Lemma \ref{lem:poles} therefore implies that for every \(\lambda_1,\lambda_2\),
 the specializations
\begin{equation}\label{eq:spec}
\big\langle {\widehat{\Hr}^{\GW}_\mu}(\mathsf{p})\big|\lambda_1,\lambda_2,\emptyset\big\rangle^{\GW,\mathsf{T}}_{U,D}\Big|_{s_3=-s_i} \quad \text{and} \quad
\big\langle \Hr^\PT_{\mu}(\mathsf{p})\big|\lambda_1,\lambda_2,\emptyset\big\rangle^{\GW,\mathsf{T}}_{U,D}\Big|_{s_3=-s_i}
\end{equation}
are well defined for both $i=1$ and $i=2$.

Consider first the $i=1$ case.
By Corollary \ref{coco9}, we have
$$
\mathcal{T} =
\underline{\mathcal{T}}|_{s_3=-s_1}\, $$
since the specialization $s_3=-s_1$ implies $c_3=c_1c_2$.
From \eqref{lldd4}, we conclude
$$ \Big\langle\, \widehat{\Hr}^{\mathsf{GW}}_{\mu}(\mathsf{p})\, \Big|\, \lambda_1,\lambda_2,\emptyset\, 
\Big\rangle^{\mathsf{GW},\mathsf{T}}_{U,D}=
q^{-|\lambda_1|-|\lambda_2|}\Big\langle\, \Hr^{\PT}_{\mu}(\mathsf{p})\, \Big|\, 
\lambda_1,\lambda_2,\emptyset\, \Big \rangle^{\PT,\mathsf{T}}_{U,D}\mod (s_1+s_3)\, .$$
By considering the $i=2$ case, we obtain the above equality
mod $(s_2+s_3)$ also. \qed

\subsection{Proof of Theorem \ref{thm:noneq}}
\label{sec:proof-theor-refthm}

Theorem \ref{thm:noneq}
follows almost immediately from the following reformulation of \cite[Theorem 7]{PPDC}.
We define
\[\widetilde{\underline{\Hr}}_\mu=\frac{1}{(c_3)^{l-1}}\sum_{\mbox{set partitions
$P$ of}\{1,\dots,l\}}(-1)^{|P|-1}(|P|-1)!\prod_{S\in P}
\underline{\mathcal{T}}(\Hr^{\PT}_{\mu_S})\, .\]
For classes \(\gamma_i\in H^*(X)\)  and a vector \(\vec{k}\) of
non-negative integers, we define 
\[\overline{\underline{\Hr}_{k_1}(\gamma_1)\dots
\underline{\Hr}_{k_l}(\gamma_l)}=\sum_{\mbox{set partitions $P$ of} \{1,\dots,l\}}\prod_{S\in P}\widetilde{\underline{\Hr}}_{\vec{k}_S}(\gamma_S),\]
where \(\gamma_S=\prod_{i\in S}\gamma_i\).

The  argument from the section 7.4 of \cite{PPDC} implies

\begin{theorem}\cite{PPDC}
Let    \(X\) be a nonsingular projective toric 3-fold,  and let
 \(\gamma_i\in H^*(X,\CC)\). After the change of variables
$-q=e^{iu}$, we have
  \[
\big\langle\overline{ \underline{\Hr}_{k_1}(\gamma_1)\dots \underline{\Hr}_{k_l}(\gamma_l)} \big \rangle^{\GW}_\beta
=
\big\langle \Hr^{\PT}_{k_1}(\gamma_1)\dots \Hr^{\PT}_{k_l}(\gamma_l)\big  \rangle^{\PT}_\beta
  \, ,\]
where the non-equivariant limit is taken on both sides.
\end{theorem}

Theorem \ref{thm:noneq} follows because $$\widetilde{\Hr}_\mu=
\widetilde{\underline{\Hr}}_\mu|_{c_3=c_1c_2}\,  $$
and the restriction $c_3=c_1c_2$ does not affect the non-equivariant
limit if all $\gamma_i$ have positive cohomological degree. \qed

\subsection{Examples for \(X=\mathbf{P}^3\)} \label{sec:examples}
 The prefactor in front of \(\sum_{k=0}^\infty x^k\ch_k(\mathbb{F}) \) in the
definition of \(\Hr^{\PT}(x)\) in Section \ref{vvtt3} has an expansion
that starts as:
\[
\mathcal{S}\left(\frac{x}{\theta}\right)=
1-\frac{c_2}{24}x^2+\frac{c_2^2}{1920}x^4-\frac{c_2^3}{322560}x^6+\dots.\]
In particular, the non-equivariant limit of \(\Hr_k^{\PT}(\gamma)\) is equal to
\[\text{ch}_{k+1}(\gamma)-\frac{1}{24}\text{ch}_{k-1}(\gamma\cdot c_2)\, .\]

On the Gromov-Witten side of the correspondence, we have
\begin{eqnarray*}
\big \langle \Hr_1^{\GW}(\gamma)\Phi\big\rangle&=&
\big\langle\mathfrak{a}_1(\gamma)\Phi\big\rangle\, ,\\
 \big\langle \Hr_2^{\GW}(\gamma)\Phi\big\rangle&=&\frac12\big\langle \mathfrak{a}_2(\gamma)\Phi\big\rangle\, ,\\
\big\langle \Hr_3^{\GW}(\gamma)\Phi\big\rangle&=&\frac16\big\langle \mathfrak{a}_3(\gamma)\Phi\big\rangle+\frac{1}{24u^2}\big\langle c_1^2c_2\cdot\Phi\big\rangle\, , \\
\big\langle \Hr_4^{\GW}(\gamma)\Phi\big\rangle&=&
  \frac1{24}\big\langle \mathfrak{a}_4(\gamma)\Phi\big\rangle-\frac{i}{12 u}
\big\langle\mathfrak{a}_1^2(c_1\cdot\gamma)\Phi\big\rangle
-\frac{5i}{144u^3}\big\langle c_1^3c_2\cdot\Phi\big\rangle\, ,\\
\big \langle\Hr_5^{\GW}(\gamma)\Phi\big\rangle&=&\frac1{120}\big\langle\mathfrak{a}_5(\gamma)\Phi\big\rangle-\frac{i}{24u}\big\langle\mathfrak{a}_1\mathfrak{a}_2(c_1\cdot \gamma)\Phi\big\rangle
  -\frac{1}{48u^2}\big\langle\mathfrak{a}^2_1(c_1^2\cdot\gamma)\Phi\big\rangle\\
 & & +\frac{1}{24u^2}\big\langle\mathfrak{a}_1(c_1^2c_2\cdot\gamma)\Phi\big\rangle-\frac{1}{64u^4}\big\langle c_1^4c_2\cdot \Phi\big\rangle\, .
\end{eqnarray*}
The operators \(\mathfrak{a}_k\) are expressed in terms of standard descendents by inverting \eqref{ff556}:
\begin{eqnarray}\label{eq:aak}
  \mathfrak{a}_1&=&\tau_0-c_2/24\,, \\ \nonumber
\frac{iu}{2}\mathfrak{a}_2&=&\tau_1+c_1\cdot\tau_0\, ,\\ \nonumber
 -\frac{u^2}{3}\mathfrak{a}_3&=&2\tau_2+3c_1\cdot\tau_1+c_1^2\cdot\tau_0\, ,\\ \nonumber
-\frac{iu^3}{4}\mathfrak{a}_4&=&6\tau_3+11c_1\cdot\tau_2+6c_1^2\tau_1+c_1^3\cdot\tau_0\,,
\\ \nonumber
 \frac{u^4}{5}\mathfrak{a}_5&=&24\tau_4+50c_1\cdot\tau_3+35c_1^2\cdot\tau_2+10c_1^3\cdot\tau_1+c_1^4\cdot\tau_0\, .
\end{eqnarray}

In particular, the $\GW/\PT$ correspondence of Theorem \ref{thm:noneq}  gives the following 
relations for the degree 1 invariants of \(\mathbf{P}^3\):
\begin{eqnarray}\label{eq:ch5L}
  i q^{-2}\big\langle \text{ch}_5(\mathsf{L})\big\rangle^\PT_1 &=&
    \frac1{u^3}\big\langle\tau_3(\mathsf{L})\big\rangle^\GW_1+
\frac{22}{3u^3}\big\langle\tau_2(\mathsf{p})\big\rangle^\GW_1
-\frac1{3u}\big\langle\tau_0\tau_0(\mathsf{p})\big\rangle^\GW_1\, ,\\ \nonumber
  -q^{-2}\big\langle \text{ch}_6(\mathsf{H})\big\rangle^\PT_1+
-\frac{q^{-2}}{4}\big\langle \text{ch}_4(\mathsf{p})\big\rangle^\PT_1 &=&
 \frac{1}{u^4}\big\langle\tau_4(\mathsf{H})\big\rangle^\GW_1
+\frac{25}{3u^4}\big\langle\tau_3(\mathsf{L})\big\rangle^\GW_1+
    \frac{70}{3u^4}\big\langle\tau_2(\mathsf{p})\big\rangle^\GW_1 \\ \nonumber
& &      -\frac{1}{3u^2}\big\langle\tau_0\tau_1(\mathsf{L})\big\rangle^\GW_1
+\frac{5}{3u^2}\big\langle\tau_0\tau_0(\mathsf{p})\big\rangle^\GW_1\, .
  \end{eqnarray}
Here, and below in Section \ref{sec:concluding-remarks},  
$$\mathsf{p}\,, \  \mathsf{L}\, ,  \ \mathsf{H} \ \in H^*(\mathbf{P}^3)$$
are respectively the classes of a point,  a line, and a plane.
These formulas can verified  numerically up to \(u^8\) with the help of Gathmann's Gromov-Witten code and previously known complete
calculations on the stable pairs
side \cite{P}.

\section{Concluding remarks: ${\mathsf{DT}}/\PT/\GW$}
\label{sec:concluding-remarks}

\subsection{Stationary ${\mathsf{DT}}/\PT$ correspondence}
\label{sec:stat-dtpt-corr}
The moduli space  \(I_n(X,\beta)\) parameterizes flat families of ideal sheaves \(\mathcal{I}\subset \mathcal{O}_X\) with
 \[\chi(\mathcal{I})=n,\quad [{\text{Supp}}(\mathcal{O}_X/\mathcal{I})]=\beta\in H_2(X,\ZZ)\, .\]
There is a universal quotient sheaf \(\mathbb{F}_n\) over \(X\times I_n(X,\beta)\) with fibers 
$$\mathbb{F}_n|_{\mathcal{I}\times X}=\mathcal{O}_X/\mathcal{I}\, .$$ We define
\[\ch_k(\gamma)=\pi_*\left(\ch_k(\mathbb{F})\cdot \gamma\right)
  \in \bigoplus_{n\in \mathbb{Z}} H^*(I_n(X,\beta))\ \ \ \text{for}\ \ \  \gamma\in H^*(X)\, .\]

The moduli space \(I_n(X,\beta)\) has a natural virtual cycle 
\([I_n(X,\beta)]^{vir}\). The integrals of the above descendents classes 
define generating series,
\[\big\langle \mathrm{ch}_{k_1}(\gamma_1)\dots\mathrm{ch}_{k_m}(\gamma_m)\big\rangle_\beta^{\mathsf{DT}}=\sum_{n\in \mathbb{Z}}q^n\int_{[I_n(X,\beta)]^{vir}}\mathrm{ch}_{k_1}(\gamma_1)\dots\mathrm{ch}_{k_m}(\gamma_m)\, ,\]
just as for stable pairs. 
The normalized generating series of ${\mathsf{DT}}$ invariants have better
properties:
\begin{equation}
\label{p99p}
\big\langle \mathrm{ch}_{k_1}(\gamma_1)\dots\mathrm{ch}_{k_m}(\gamma_m)\big\rangle_\beta^{\mathsf{DT}'}=\big\langle\mathrm{ch}_{k_1}(\gamma_1)\dots\mathrm{ch}_{k_m}(\gamma_m)\big\rangle_\beta^{\mathsf{DT}}/\langle 1\rangle_0^{\mathsf{DT}}\, . \end{equation}
We refer the reader to \cite{MNOP1,MNOP2} for a
more detailed introduction.

The argument of \cite{PPDC} is valid if we replace the Gromov-Witten
side by the ${\mathsf{DT}}$ theory of ideal sheaves.
Since the 1-leg invariants in ${\mathsf{DT}}$ and $\PT$ theories are identical modulo \(s_1+s_2\), our proof of Theorem \ref{thm:noneq} can be repeated to obtain the following non-equivariant
result.
\begin{theorem}
Let \(X\) be  a nonsingular projective toric 3-fold, and
let \(\gamma_i\in H^{\geq 2}(X,\mathbb{C})\). The 
stationary descendent
${\mathsf{DT}}/\PT$ correspondence
holds: 
  \[
\big\langle \mathrm{ch}_{k_1}(\gamma_1)\dots \mathrm{ch}_{k_l}(\gamma_l)\big\rangle^{\mathsf{DT}'}_{\beta}=
\big\langle \mathrm{ch}_{k_1}(\gamma_1)\dots \mathrm{ch}_{k_l}(\gamma_l)\big  \rangle^{\PT}_{\beta}\, .
  \]
\end{theorem}

Finding a relation between   ${\mathsf{DT}}$ and $\PT$ 
theories which includes the descendents of the 
identity class \(1\) is more subtle. 
A basic source of difficulty is that the
$\mathsf{DT}$ generating series \eqref{p99p} are {\em not} always
rational functions. We expect
 the non-equivariant ${\mathsf{DT}}$ descendent series
 to depend upon the function \(F_3\),
$$F_3(q)=\sum_{i=1}^\infty n^2\frac{q^n}{1-q^n}\, , $$
 which arises as
the logarithmic derivative of the McMahon function.

%and by $\mathcal{R}$ the field of rational functions in $q$.

\begin{conjecture}\label{conjDT} Let $X$ be a nonsingular projective 3-fold.
For $\gamma_i\in H^*(X,\mathbb{C})$,
the series 
$$ \big\langle \mathrm{ch}_{k_1}(\gamma_1)\dots \mathrm{ch}_{k_m}(\gamma_m)\big\rangle_{\beta}^{\mathsf{DT}^\prime}$$ 
is  a polynomial in $(q\frac{d }{d q})^i F_3(-q)$ for \(0\le i\le m\) with coefficients in the ring of rational functions of \(q\).
\end{conjecture}

\subsection{Beyond the stationary case}
\label{sec:beyond-stat-case}
The $\GW/\PT$ correspondence for the complete non-equivariant descendent
theory is expected to be significantly more complex than
the stationary case because
of the analytic properties of the $\GW$ descendent series. 
In fact, we expect the analytic complexities
of the $\GW$ descendent series to be very similar 
to those of the  $\mathsf{DT}$ descendent series. 

The Euler-Maclaurin formula provides an asymptotic expansion of $F_3(-q)$ at $u=0$: 
$$ F_3(u)\, \sim\,  2\zeta(3)/u^3-\sum_{n=0}^\infty\frac{B_{2n+2}B_{2n}}{(2n)!(2n+2)}(iu)^{2n-1}\, .$$
For  simplicity, we will use the notation $F_3(u)$ for the  part of 
the expansion without the most singular term. In order words, we define
$$ F_3(u)=-\sum_{n=0}^\infty\frac{B_{2n+2}B_{2n}}{(2n)!(2n+2)}(iu)^{2n-1}\, .$$

Let us denote by \(\mathcal{R}\) the ring of rational functions of \(q=-e^{iu}\). The following is a Gromov-Witten version of 
Conjecture~\ref{conjDT}.

%\begin{conjecture}\label{conjGW} For any collection $\alpha_1,\dots,\alpha_m$ of cohomology classes of $X$ the invariant
%$$ \langle \tau_{k_1}(\alpha_1)\dots \tau_{k_m}(\alpha_m)\rangle^{\GW}_{\beta}$$ 
%is  a polynomial of $(\frac{d }{d u})^m F_3(u)$ with coefficients in the ring \%(\mathcal{R}[u^{\pm 1}]\) for any choice of numbers $k_1,\dots,k_m.$
%\end{conjecture}

\begin{conjecture}\label{conjGW} Let $X$ be a nonsingular projective 3-fold.
For $\gamma_i\in H^*(X,\mathbb{C})$,
the series 
$$ \big\langle \tau_{k_1}(\gamma_1)\dots \tau_{k_m}(\gamma_m)\big\rangle^{\GW}_{\beta}$$ 
is  a polynomial in $(\frac{d }{d u})^i F_3(u)$ for
$0\le i \le m$ with coefficients in the ring \(\mathcal{R}[u^\pm]\).
\end{conjecture}

The power series $F_3(q)$ does not converge at any point of a circle $|q|=1$.
 Hence, the $q$-derivatives of 
$F_3(q)$ are linearly independent over $\mathcal{R}$. 
Otherwise $F_3(q)$ would be a solution of 
a non-trivial linear 
differential equation with rational coefficients and hence  analytic outside finite number of points.
We can therefore
define a homomorphism $$\Theta: \mathcal{R}\left[u^{\pm 1},F_3(-q),q\frac{d}{d q}F_3(-q),\dots\right]\to
\mathbb{Q}[u^{-1}][[u]].$$ 
On the subring $\mathcal{R}[u^{\pm 1}]$, the homomorphism $\Theta$ is defined by  the change of variable $e^{iu}=-q$, and, on the generators $(q\frac{d}{dq})^m F_3(-q)$, $\Theta$ is defined by:
$$\left( q\frac{d}{d q}\right)^m F_3(-q)\, \mapsto\,  \frac{1}{2}\left(-i\frac{d}{d u}\right)^m F_3(u)\, .$$
Note the factor of \(\frac{1}2\) in the last formula:
the homomorphism \(\Theta\) is {\em not} merely a change of variable.

\begin{conjecture} \label{gg559} 
We have
$$ 
\Theta\left(\big\langle\prod_i \Hr^{\mathrm{DT}}_{k_i}(\mathsf{p})\big|\mu_1,\mu_2,\mu_3\big\rangle^{\mathsf{DT}'}_{U,D}\right) = \big\langle\prod_i \Hr^{\GW}_{k_i}(\mathsf{p})|\mu_1,\mu_2,\mu_3\big\rangle^{\GW}_{U,D}\ 
\mod (c_1c_2-c_3)^2.$$
\end{conjecture}

\noindent If Conjecture \ref{gg559} were true, then we could also write a conjecture
for the complete
non-equivariant $\GW/\mathsf{DT}$ descendent correspondence
via the formulas  of Section \ref{sec:non-equiv-limit}. 
The appearance of \(\frac{1}{2}\) in the definition of \(\Theta\) could be motivated by the computation
of degree \(0\) \(\GW\) and \(\mathsf{DT}\) invariants \cite{MNOP1}. Though
Conjecture \ref{gg559} is  mysterious, the claimed equality has been supported
by a large number of numerical experiments.

%In the introduction we explained how one can  combine the machinery of \cite{PPDC} to  imply GW/PT correspondence for toric varieties in nonequivariant setting.
%We expect that the same formulas provide the non-stationary GW/DT correspondence:

%\begin{conjecture}\label{gwdtwith1} 
%For any  
%$\gamma_1,\dots,\gamma_m\in H^{*}(X)$ and $\beta\in H_2(X,\ZZ)$ after the change of variable 
%$e^{iu}=-q$ we have
%\begin{equation*}
%\langle \Hr^{\mathrm{DT}}_{k_1}(\gamma_1)\dots\Hr^{\mathrm{DT}}_{k_m}(\gamma_m)\rangle^{\mathrm{DT}}_{X}\to \langle \overline{\Hr_{k_1}(\gamma_1)\dots\Hr_{k_m}(\gamma_m)}\rangle^{\GW}_{X}
%\end{equation*}
%\end{conjecture}

\subsection{Equivariant DT series} Studying the $\GW/{\mathsf{DT}}$ descendent 
correspondence in the general $\mathsf{T}$-equivariant setting is  difficult
for many reasons.{\footnote{The existence of a
$\mathsf{T}$-equivariant $\GW/{\mathsf{PT}}$ descendent
correspondence is proven in \cite{PPDC}, but closed formulas
are not known.}} 
 A major unresolved question concerns the analytic properties 
of the generating series for $\mathsf{T}$-equivariant Gromov-Witten descendent 
invariants. However, based on computer experiment, we 
propose conjectures controlling the behavior of 
the $\mathsf{T}$-equivariant ${\mathsf{DT}}$ descendent
series.

Let $X$ be a nonsingular projective toric 3-fold equipped with
an action of the 3 dimensional torus $\mathsf{T}$.
As in Section \ref{ccppvv}, let
$$H_{\mathsf{T}}(\bullet)= \mathbb{C}[s_1,s_2,s_3]\, .$$
We define the algebra $\Fr$ generated by the series{\footnote{For fun, we
had originally termed these {\em Frankenstein series}.}}
$$
F_{2k+1}(-q) = \sum_{n=0}^\infty (-q)^n \sum_{d|n} d^{2k}\,, \quad k\ge 1 \,,
$$
\footnote{\(n=0\) term is defined by renormalization.}and their iterated \(q\frac{d}{dq}\) derivatives.

\begin{conjecture} \label{DT0}
The $\mathsf{T}$-equivariant $\mathsf{DT}$ descendent
series of $X$ satisfy
$$\big\langle  
\mathrm{ch}_{k_1}(\gamma_1)\dots \mathrm{ch}_{k_l}(\gamma_l)
\big\rangle^{\mathsf{DT},\mathsf{T}}_\beta\ \in H^*_{\mathsf{T}}(\bullet) 
\otimes \Q(q) \otimes \Fr\, $$
for $\gamma_i\in H^*_{\mathsf{T}}(X,\mathbb{C})$ and $\beta\in H_2(X,\mathbb{Z})$.
\end{conjecture}

Conjecture~\ref{DT0} fits into a web of conjectures about the analytic behavior
of generating functions of equivariant integrals of  tautological classes
over moduli spaces of sheaves \cite{OkZ}. We refer the reader to \cite{OkZ} for
more motivation, further conjectures, and future directions.

\vspace{-10pt}
\Addresses
\end{document}